\documentclass[a4paper,11pt,oneside]{amsart}
\usepackage[left=2.7cm,right=2.7cm,top=3.5cm,bottom=3cm]{geometry}

\usepackage[all]{xy}
\usepackage{amscd,mathrsfs}
\usepackage{amsfonts}
\usepackage{amssymb}
\usepackage{amsmath}
\usepackage{latexsym}
\usepackage[usenames,dvipsnames]{color}

\newcommand{\C}{\mathbb{C}}
\newcommand{\F}{\mathbb{F}}
\newcommand{\I}{\mathbb{I}}
\newcommand{\N}{\mathbb{N}}
\newcommand{\Q}{\mathbb{Q}}

\newcommand{\Z}{\mathbb{Z}}
\newcommand{\bbS}{\mathbb{S}}

\renewcommand{\k}{\kappa}
\newcommand{\e}{\varepsilon}
\newcommand{\G}{\Gamma}
\newcommand{\g}{\gamma}
\renewcommand{\L}{\Lambda}

\renewcommand{\leq}{\leqslant}
\renewcommand{\geq}{\geqslant}
\newcommand{\ds}{\displaystyle}

\newcommand{\liminv}{\displaystyle \lim_{\leftarrow}}

\newcommand{\calc}{\mathcal C}
\newcommand{\calf}{\mathcal F}
\newcommand{\cali}{\mathcal I}
\newcommand{\calm}{\mathcal M}
\newcommand{\calr}{\mathcal R}
\newcommand{\CaCl}{\mathcal{C}\ell}
\newcommand{\calCl}{\mathcal{C}\ell^{\,0}}

\newcommand{\pr}{\mathfrak p}
\newcommand{\gota}{\mathfrak a}
\newcommand{\gotb}{\mathfrak b}

\newcommand{\gotq}{\mathfrak q}
\newcommand{\gotI}{\mathfrak I}
\newcommand{\gotM}{\mathfrak M}

\newcommand{\bmu}{\boldsymbol\mu}
\newcommand{\sri}{\twoheadrightarrow}
\newcommand{\iri}{\hookrightarrow}

\newcommand{\ov}{\overline}
\newcommand{\wt}{\widetilde}
\newcommand{\wh}{\widehat}

\newcommand{\il}[1]{\lim_{\buildrel \longleftarrow\over{#1}}}
\newcommand{\plim}[1]{\displaystyle{\lim_{\stackrel{\longleftarrow}{#1}}}\,}

\DeclareMathOperator{\Gal}{Gal}
\DeclareMathOperator{\Hom}{Hom}
\DeclareMathOperator{\Fitt}{Fitt}
\DeclareMathOperator{\Ext}{Ext}
\DeclareMathOperator{\Inf}{Inf}
\DeclareMathOperator{\Ker}{ker}
\DeclareMathOperator{\Fr}{Fr}

\theoremstyle{plain}
\newtheorem{thm}{Theorem}[section]
\newtheorem{cor}[thm]{Corollary}
\newtheorem{lem}[thm]{Lemma}
\newtheorem{prop}[thm]{Proposition}
\newtheorem{quest}[thm]{Open Question}

\theoremstyle{definition}
\newtheorem{defin}[thm]{Definition}
\newtheorem{rem}[thm]{Remark}
\newtheorem{rems}[thm]{Remarks}

\title[IMC for the Carlitz cyclotomic extension and applications]
{Iwasawa main conjecture for the Carlitz cyclotomic extension and applications}

\author[B. Angl\`es] {Bruno Angl\`es}
\address{Bruno Angl\`es: LMNO, Universit\'e de Caen BP 5186, 14032 Caen Cedex. France} \email{bruno.angles@unicaen.fr }

\author[A. Bandini] {Andrea Bandini}
\address{Andrea Bandini: Dipartimento di Matematica e Informatica, Universit\`a degli Studi di Parma\\
Parco Area delle Scienze, 53/A - 43124 Parma (PR), Italy}
\email{andrea.bandini@unipr.it}

\author[F. Bars] {Francesc Bars}% \maltese\protect\footnote{\maltese Supported by MTM2009-10359}}
\address{Francesc Bars: Departament de Matem\`atiques, Facultat de Ciencies, Universitat Aut\`onoma de Barcelona\\
08193 Bellaterra (Barcelona), Catalonia}
\email{francesc@mat.uab.cat}
\thanks{F. Bars supported by by MTM2013-40680-P.}

\author[I. Longhi]{Ignazio Longhi}
\address{Ignazio Longhi: Department of Mathematical Sciences,
Xi'an Jiaotong-Liverpool University \\
111 Ren Ai Road, Dushu Lake Higher Education Town\\
Suzhou Industrial Park, Suzhou, Jiangsu\\
215123, China}
\email{Ignazio.Longhi@xjtlu.edu.cn}

\begin{document}

\begin{abstract}
We prove an Iwasawa Main Conjecture for the class group of the $\pr$-cyclo\-to\-mic extension $\calf$ of the function field
$\F_q(\theta)$ ($\pr$ is a prime of $\F_q[\theta]$), showing that its Fitting ideal is generated by a Stickelberger element.
We use this and a link between the Stickelberger element and a $\pr$-adic $L$-function to prove a close analog of the
Ferrero-Washington Theorem for $\calf$ and to provide information on the $\pr$-adic valuations of the Bernoulli-Goss numbers
$\beta(j)$ (i.e., on the values of the Carlitz-Goss $\zeta$-function at negative integers).
\end{abstract}

\keywords{Stickelberger series; Carlitz-Goss $\zeta$-function;
$L$-functions;  Bernoulli-Goss numbers; class groups; Iwasawa Main
Conjecture; function fields; cyclotomic extensions.}

\subjclass[2010]{11R60; 11R23; 11M38; 11S40; 11R58}

\maketitle

\section{Introduction}\label{sec:intro}
Let $F=\F(\theta)$ (where $\F$ is a finite field of cardinality $q=p^e\geqslant 3$, with $p$ prime\,\footnote{We need $q>2$ in
order to have non-trivial characters of the group $\Delta$ (see below).}) and fix as $\infty$ the place corresponding to
$\frac{1}{\theta}$. The ring of functions regular outside $\infty$ is
$A=\F[\theta]$ and we denote by $A_+$ the set of monic polynomials of $A$. Consider the {\em Carlitz-Goss $\zeta$-function}
\[ \zeta_A(s):=\sum_{a\in A_+} a^{-s}\ \ , \ s\in \C_\infty^*\times \Z_p \]
(see Section \ref{CGBG} for the precise definition of all terms). This
function represents the function field version of the Riemann $\zeta$-function for number fields: its values have analogous
relevant arithmetical interpretations and have been the subject of investigations in recent years (see, for example, Taelman's results in
\cite{Ta1} and \cite{Ta2}). There is a canonical embedding $\Z \iri\C_\infty^*\times \Z_p$ sending $j$ to $(\theta^j,j)$ and we usually
write $\zeta_A(j)$ for the value $\zeta_A(\theta^j,j)$.
Our main goal is to provide some information on special values at negative integers by studying the associated {\em Bernoulli-Goss
numbers} $\beta(j)\in A$ (see Definition \ref{DefZXjBGPol}). To do this we first prove an Iwasawa Main Conjecture for the
$\pr$-cyclotomic extension $\calf$ of $F$ ($\pr$ a prime of $A$), which will provide a link between a $\pr$-adic $L$-function
interpolating the values of $\zeta_A$ and a Stickelberger element related to class groups of subextensions of $\calf$.

Let $\Phi$ be the Carlitz module associated with $A$, fix a prime $\pr$ of $A$ and denote by $\Phi[\pr^r]$ the $\pr^r$-torsion
of the Carlitz module (see Section \ref{Not} or \cite[Chapter 12]{Ro}). Let
$F_n:=F(\Phi[\pr^{n+1}])$ and $\calf:=\cup_n F_n$. Then $\calf/F$ is a Galois extension ramified only at $\pr$ and $\infty$,
whose Galois group is isomorphic to $\Delta\times \Z_p^\infty$, with $\Delta$
cyclic of order $q^d-1$, where $d:=\deg(\pr)$. In Section \ref{SecTwoChar} we are going to define two Teichm\"uller characters
of $\Delta$, named $\omega_\pr$ and $\wt{\omega}_\pr$, taking values in characteristic $p$ and $0$ respectively. For any $y\in
\Z_p$ and $0\leq i\leq q^d-2$ there exists a $\pr$-adic $L$-function $L_\pr(X,y,\omega_\pr^i) \in A_\pr[[X]]$, which
interpolates special values of $\zeta_A$; indeed there exist polynomials $Z(X,i) \in A[X]$ such that
\begin{equation}\label{EqIntro1}
Z(1,i)=\zeta_A(-i) \quad {\rm and} \quad L_\pr(X,y,\omega_\pr^i) \equiv Z(X,i) \pmod{\pr}
\end{equation}
(see Section \ref{SecprLfun}).

On the algebraic side we study $\calCl(F_n)\{p\}$, that is, the $p$-part of the group of classes of degree zero divisors of $F_n$.
Let $X_{F_n}$ be the curve associated with the field $F_n$ and denote by $T_p(F_n):=T_p(Jac(X_{F_n})(\ov{\F}))$
the $p$-adic Tate module of its Jacobian;
then $\calCl(F_n)\{p\}$ can be recovered as $G_\F$-coinvariants of $T_p(F_n)$ (see Lemma \ref{FromTpToCn}; here $G_\F:=\Gal(\ov{\F}/\F)$,
with $\ov{\F}$ a fixed algebraic closure of $\F$).
Let $\chi$ be any character defined on $\Delta$ (hence $\chi=\wt{\omega}_\pr^i$ for some $0\leq i\leq q^d-2$) and
let $e_\chi$ be the associated idempotent: for any module $M$ we shall denote by $M(\chi):=e_\chi M$ the $\chi$-component of $M$.
Note that all these characters have values in $W$, the Witt ring of the residue field $\F_\pr$ at $\pr$, so, from now on, we consider
algebras over the ring $W$. Moreover we recall that such characters are called {\em even} if $q-1$
divides $|\ker(\chi)|$ and {\em odd} otherwise.

Define the {\em Stickelberger series} for $F_n/F$ as
\[ \Theta_n(X):= \prod_{\gotq\in \mathscr{P}_F-\{\pr,\infty\}} (1-\Fr_\gotq^{-1}X^{\deg(\gotq)})^{-1} \ ,\]
where $\mathscr{P}_F$ is the set of places of $F$ and $\Fr_\gotq\in\Gal(F_n/F)$ is the Frobenius at $\gotq$.
For any character $\chi$ we put $\Theta_n(X,\chi)$ for the series verifying $\Theta_n(X,\chi)e_\chi=e_\chi\Theta_n(X)$.

\noindent Using some recent results of C. Greither and C.D. Popescu (see \cite{GP1} and \cite{GP2}) we are able to compute
the Fitting ideals of the Tate modules over the algebra $W[\Gal(F_n/F_0)][[G_\F]]$ (see Proposition \ref{FittTFn}).

\begin{prop} \label{PropIntro2}
We have
\begin{equation}\label{EqIntro3} \Fitt_{W[\Gal(F_n/F)][[G_\F]]} \big((W\otimes T_p(F_n))(\chi)\big) =
(\Theta_n^\#(\g^{-1},\chi)) \ ,\end{equation}
where
\[\Theta_n^\#(\g^{-1},\chi):= \left\{ \begin{array}{cl} \Theta_n(\g^{-1},\chi) & {\rm if}\ \chi\ {\rm is\ odd} \\
\ & \\
\frac{\Theta_n(\g^{-1},\chi)}{1-\g^{-1}} & {\rm if}\
\chi\neq \chi_0 \ {\rm is \ even} \end{array} \right. \ ,\]
and $\g$ is the arithmetic Frobenius in $G_\F\,$.
\end{prop}

\begin{rem}\label{RemIntro1}
The trivial character $\chi_0$ needs (as usual) some special treatment due to the presence of the ramified place $\pr$
in the $\chi_0$-component of the relevant Iwasawa modules. In this paper we shall provide some results on the $\chi_0$-component
as well but we leave the precise statements to the following sections (see, in particular, Corollary \ref{FittTrivChar}
and Remark \ref{RemChi0}),
since they require a few more notations and all the main arithmetical applications will involve only the $\chi$-components
for $\chi\neq \chi_0$.
\end{rem}

Then, in Section \ref{SecFittInf}, we prove all the compatibility relations needed to compute inverse limits in equality
(\ref{EqIntro3}). Putting $\Lambda:=W[[\Gal(\calf/F_0)]]$, we obtain (see Theorem \ref{FittCalmCalf})

\begin{thm}\label{IntroIMC1}
For any $\chi\neq \chi_0$, $T_p(\calf)(\chi):=\plim{n} T_p(F_n)(\chi)$ is a finitely generated torsion $\Lambda[[G_\F]]$-module and
\[ \Fitt_{\Lambda[[G_\F]]}\left( T_p(\calf)(\chi) \right) =
\il{n} (\Theta^\#_n(\gamma^{-1},\chi)) = :(\Theta^\#_\infty(\gamma^{-1},\chi)) \ . \]
\end{thm}

To compute the Fitting ideals of the class groups (i.e., of the $G_\F$-coinvariants of the Tate modules) one simply
specializes $\g^{-1}$ to 1. In Section \ref{SecIMC} (in particular in Lemma \ref{BigLemma}) we again check the compatibility relations
needed to be able to perform (inverse) limits after this specialization. Thus we prove an {\em Iwasawa Main
Conjecture} (IMC) for the $\pr$-cyclotomic extension $\calf/F$ as an
equality of ideals in the (non-noetherian) Iwasawa algebra $\Lambda$; namely (see Theorem \ref{IMClevel})

\begin{thm}[IMC]\label{IntroIMC}
For any $\chi\neq \chi_0$, $\calc(\calf)(\chi):=\plim{n} \calCl(F_n)\{p\} (\chi)$ is a finitely generated torsion
$\Lambda$-module, and
\[ \Fitt_{\Lambda}\left( \calc(\calf)(\chi) \right) = \il{n} (\Theta^\#_n(1,\chi)) = (\Theta^\#_\infty(1,\chi)) \ . \]
\end{thm}

In Section \ref{SecStick} we show the interpolation properties of the Stickelberger series $\Theta_{\calf_S/F,S}(X)$
(where $\calf_S$ is the maximal abelian extension of $F$ unramified outside a prescribed finite set of places $S$). This
series turns out to be a kind of universal object for $L$-functions attached to abelian characters unramified outside $S$.
If $\mathbb{L}$ is any of the fields $\C$, $\C_\infty$ or $\C_\pr$ and $\psi$ is a continuous character
from $G_S:=\Gal(\calf_S/F)$ to $\mathbb{L}^*$, then we provide relations between $\psi(\Theta_{\calf_S/F,S})$ and
the complex $L$-function $L_S(s,\psi)$ (for $\mathbb{L}=\C$),
the Carlitz-Goss $\zeta$-function $\zeta_A$ (for $\mathbb{L}=\C_\infty$) and
the $\pr$-adic $L$-function $L_\pr(X,\psi)$ (for $\mathbb{L}=\C_\pr$). Details and more precise formulations will be
given after introducing the proper notations,
in equation \eqref{eqls}, Theorem \ref{thmtetzet} and Theorem \ref{thmtetpr} respectively.

Exploiting these relations and the fact that the Stickelberger series $\Theta^\#_\infty(X,\chi)\in\Lambda[[X]]$ is (more or less)
the projection of $e_\chi\Theta_{\calf_S/F,S}(X)$, we provide the final link between $\Theta^\#_\infty(X,\chi)$ and the $\pr$-adic
$L$-function $L_\pr(X,y,\chi)$ using a theorem of Sinnott (see \cite[Theorem 1]{Si})
which, in our interpretation, provides an injection $s\colon\Lambda/p\Lambda \iri Dir(\Z_p,A_\pr)$ (a certain ring of
continuous functions $\Z_p\rightarrow A_\pr$, defined in Section \ref{SecSinn} below). Extending
the map $s$ in a natural way to $s_X\colon\Lambda[[X]]\rightarrow Dir(\Z_p,A_\pr)[[X]]$, we obtain (see Theorem \ref{StickLfun})

\begin{thm}\label{ThmIntro1}
Let $0\leq i\leq q^d-2$, for any $y\in\Z_p$, we have
\[ s_X(\Theta_{\infty}(X,\wt{\omega}_\pr^{-i}))(y)=L_{\pr}(X,-y,\omega_\pr^i) \ .\]
\end{thm}

Theorems \ref{IntroIMC} and \ref{ThmIntro1} are the key ingredients for the arithmetical applications we had in mind.
Let $\pi_\pr\in A_+$ denote a generator of $\pr$. In the final section we show that the Bernoulli-Goss numbers $\beta(j)$ verify:
\begin{itemize}
\item if $j\equiv i\pmod{q^d-1}$ and $j\not\equiv 0\pmod{q-1}$ (i.e., $\chi=\wt\omega_\pr^i$ is odd), then
\begin{equation}\label{EqIntro4}
L_{\pr}(1,j,\omega_\pr^i)=(1-\pi_\pr^j)Z(1,j)=(1-\pi_\pr^j)\beta(j) \ ;\end{equation}
\item if $j\geq 1$, $j\equiv i\pmod{q^d-1}$ and $j\equiv 0\pmod{q-1}$ (i.e., $\chi=\wt\omega_\pr^i$ is even), then
\begin{equation}\label{EqIntro5}
\frac{d}{dX}L_{\pr}(X,j,\omega_\pr^i)|_{X=1}=-(1-\pi_\pr^j)\frac{d}{dX}Z(X,j)|_{X=1}
= -(1-\pi_\pr^j)\beta(j) \ .\end{equation}
\end{itemize}

Using the fact that the $\beta(j)$ are nonzero (see Lemma \ref{BetaCongrTheta}) and the injectivity of the map $s$,
we get an analog of the Ferrero-Washington Theorem on the vanishing of the Iwasawa $\mu$-invariant for the
cyclotomic $\Z_p$-extension of an abelian number field (Theorem \ref{Sticnon0modp}).

\begin{thm}\label{IntroThmmu=0}
For any $1\leq i\leq q^d-2$, one has $\Theta_\infty^\#(1,\wt{\omega}_\pr^i)\not\equiv 0\pmod{p}$.
\end{thm}

Therefore the index
\[ N_\pr(i):=\Inf\{n\geq 0\,:\ \Theta_n^\#(1,\wt{\omega}_\pr^i)\not\equiv 0\pmod{p}\,\} \]
is well defined for any $1\leq i\leq q^d-2$. Using again the map $s_X$, we show that $N_\pr(i)$ provides a lower bound
for the $\pr$-adic valuations of the Bernoulli-Goss numbers $v_\pr(\beta(j))$ for $j\geq 1$, $j\equiv -i\pmod{q-1}$
(Corollary \ref{ArPropBGNum}): a relation that, to our knowledge, seems to have no counterpart in the number field setting.

In the appendix we generalize some results of \cite[Sections 2 and 3]{GP2} to the case in which
$X/Y$ is a finite geometric cover of curves over $\F$ of Galois group $G$, and we assume that $X/Y$
has a totally ramified place. In this setting we compute $Fitt_{\Z_p[G][[G_{\F}]]}(T_p(Jac(X)(\ov{\F}))^*)$ and
$Fitt_{\Z_p[G]}(\calCl(X)\{p\}^{\vee})$ where $^*$ (resp. $^\vee$) denotes the $\Z_p$-dual (resp. the Pontrjagin dual).

\subsection*{Aknowledgments} All the authors thank the MTM 2009-10359, which funded a workshop on Iwasawa theory for
function fields in 2010 and supported the authors during their stay in Barcelona in the summer of 2013, and the
CRM (Centre de Recerca Matem\`atica, Bellaterra, Barcelona) for providing a nice environment to work on this project.
The fourth author thanks NCTS/TPE for support to travel to Barcelona in summer 2013.

\section{Basic facts on the $\pr$-cyclotomic extension}\label{Not}
We recall here some basic facts (and notations) about what we call the $\pr$-cyclotomic extension of the rational function field,
including the corresponding Iwasawa algebra and the Iwasawa modules which will be relevant for our work.

\subsection{The $\pr$-cyclotomic extension}\label{SecCycExt}
Let $\F$, $F$, $A$ and the place $\infty$ be as in the introduction. Let $\Phi$ be the Carlitz module associated with $A$:
it is an $\F$-linear ring homomorphism
\[ \Phi:A\rightarrow F\{\tau\} \;\;, \;\; \theta \mapsto \Phi_{\theta}=\theta\tau^0+\tau \ ,\]
where $F\{\tau\}$ is the skew polynomial ring with $\tau f=f^q\tau$ for any $f\in F$.

We fix once and for all an algebraic closure of $F$, which shall be denoted by $\ov{F}$. For any ideal $\gota$ of $A$ write
\[ \Phi[\gota]:=\{x\in\ov{F}\, :\,\Phi_a(x)=0\ \forall\,a\in\gota \} \]
for the $\gota$-torsion of $\Phi$. It is an $A$-module isomorphic to $A/\gota$; in particular,
if $\gota|\gotb$ as ideals of $A$, we have $\Phi[\gotb]\supset\Phi[\gota]$.

Fix a prime ideal $\pr\subset A$ of degree $d>0$ and, for any $n\in\mathbb{N}$, let
\[ F_n:=F(\Phi[\pr^{n+1}]) \]
be the field generated by the $\pr^{n+1}$-torsion of the Carlitz module. It is well known (see \cite[Chapter 12]{Ro}
or \cite[\S 7.5]{Goss}) that $F_n/F$ is an abelian extension with Galois group
\[ G_n:= \Gal(F_n/F)=\Delta\times\G_n \simeq (A/\pr^{n+1})^* \simeq (A/\pr)^*\times (1+\pr)/(1+\pr^{n+1}) \ ,\]
where $\Delta\simeq\Gal(F_0/F)\simeq (A/\pr)^*$ is a cyclic group of order $q^d-1$, and
$\G_n=\Gal(F_n/F_0)$ is the $p$-Sylow subgroup of $G_n$. (By a slight abuse of notation, we identify the prime-to-$p$
part of $G_n$ for all $n$, and denote it always as $\Delta$.) The extension $F_n/F$ is totally ramified at
$\pr$ and tamely ramified at the place $\infty$, whose inertia group is cyclic of order $q-1$; in particular, $F_n/F_0$ is
only ramified at $\pr$. In the isomorphism $\Gal(F_0/F)\simeq (A/\pr)^*$, the inertia at $\infty$ corresponds to $\F^*$.

\begin{defin}\label{DefCycExt}
We define the {\em $\pr$-cyclotomic extension of $F$} as the field
\[ \calf:= F(\Phi[\pr^\infty]) = \bigcup_n F(\Phi[\pr^n]) \]
with abelian Galois group
\[G_\infty:= \Gal(\calf/F)= \il{n} \Gal(F_n/F) = \Delta\times \il{n} \G_n=: \Delta\times \G \ . \]
\end{defin}

For any place $v$ of $F$ we denote by $I_{v,n}$ (resp. $\cali_v$) its inertia group in  $G_n$ (resp. in $G_\infty$).
The set of ramified places in $\calf/F$ is $S:=\{\pr,\infty\}$ and, for any $n$,  one has
\[ I_{\pr,n}=G_n\ \ ,\ \ \cali_\pr=G_\infty\ \ {\rm and}\ \ I_{\infty,n}=\cali_\infty \iri \Delta \ .\]

Denote by $A_\pr$ the completion of $A$ at $\pr$, $F_\pr$ the completion of $F$ at $\pr$ and $\F_\pr$ the residue field
of $A_\pr$. Readers who prefer a more ``hands-on'' approach might appreciate the equality $A_\pr=\F_\pr[[\pi_\pr]]$,
where $\pi_\pr$ is the monic irreducible generator of $\pr$ in $A$.
The group of units of the local ring $A_\pr$ has a natural filtration; we put $U_n:=1+\pr^nA_\pr$.
Let $\mathbb{C}_{\pr}$ be the completion of an algebraic closure of $F_\pr$; we also fix once and for all an
embedding $\overline{F} \iri \mathbb{C}_\pr$.

We have isomorphisms $G_\infty\simeq A_\pr^*$ and $\G\simeq U_1$, which are induced by the {\em $\pr$-cyclotomic character}
$\kappa$. To define $\kappa$ we extend $\Phi$ to a formal Drinfeld module which we denote by the same symbol
$\Phi\colon A_\pr\rightarrow A_\pr\{\{\tau\}\}$ (see \cite{rosen}). Then for any $\sigma\in G_\infty$ and any
$\varepsilon\in\Phi[\pr^\infty]$ we have
\begin{equation} \label{eqkappa}
\sigma(\varepsilon)=\Phi_{\kappa(\sigma)}(\varepsilon) \ .\end{equation}
This action provides the isomorphisms mentioned above. In particular, $\G_n$ corresponds to
$U_1/U_{n+1}\simeq (1+\pr A_\pr)/(1+\pr^{n+1}A_\pr)$. It is well known that $U_1\simeq\Z_p^\infty$
(a product of countably many copies of $\Z_p$).

As mentioned earlier, we define $\F_\pr$ to be the residue field $A_\pr/\pr A_\pr$; it is the same as the residue field $A/\pr$.
Since we are in positive characteristic, $\F_\pr$ can be canonically identified with a subring of $A_\pr$ (by lifting $x\neq0$ to
$\tilde x$, the unique root of 1 whose reduction mod $\pr$ is $x$) and in the rest of the paper we shall generally think of it as such.

\subsection{The Iwasawa algebra}\label{SecIwaAlg} Let $W$ be the Witt ring of $\F_\pr$, which is isomorphic to $\Z_p[\bmu_{q^d-1}]$
(where $\bmu_{q^d-1}$ denotes the $(q^d-1)$-th roots of unity). By definition of Witt ring, we have an identification
$W/pW=\F_\pr$. Moreover, the projection $W\twoheadrightarrow\F_\pr$ has a partial inverse $\F_\pr^*\rightarrow\bmu_{q^d-1}$,
the {\em Teichm\"uller character} (again by lifting $x$ to $\hat x$, the unique root of 1 whose reduction mod $p$ is $x$).

The Iwasawa algebra we shall be working with is the completed group ring
\[ \Lambda:=W[[\G]]=\varprojlim W[\G_n] \ . \]
For any $n\geq 0$, put $\Gamma^{(n)}:=\Gal(\calf/F_n)$. The exact sequence $\G^{(n)}\hookrightarrow\G\twoheadrightarrow\G_n$ induces
\[ \gotI_n\hookrightarrow \Lambda\twoheadrightarrow W[\G_n]\ .\]
We also put $\gotM_n:=p^n\Lambda+\gotI_{n-1}$ for any $n\geq 1$.

We recall some other basic facts on this non-noetherian algebra:
\begin{itemize}
\item $\left\{\gotM_n\right\}_{n\geq 0}$ is a basis of neighbourhoods of zero in $\Lambda$;
\item $\Lambda/p\Lambda=\F_\pr[[\G]]$;
\item $\Lambda$ is a compact $W$-algebra and a complete local ring with maximal ideal $\gotM_1$, so that
\[ \Lambda/\gotM_1\simeq W/pW\simeq\F_{\pr}\ .\]
\end{itemize}

\subsection{Consequences of a theorem of Sinnott}\label{SecSinn}
Let $C^0(\Z_p,A_\pr)$ be the space of continuous functions from $\Z_p$ to $A_\pr$, endowed with the topology of uniform convergence.
More generally, we can consider $C^0(\Z_p,M)$ where $M$ is any finitely generated $A_\pr$-module: it turns out that $C^0(\Z_p,A_\pr)$
is the projective limit of $C^0(\Z_p,A/\pr^n)$ as $n$ varies.

Following \cite{Si}, we define the $A_\pr$-module of Dirichlet series $Dir(\Z_p,A_\pr)$ as the closure in $C^0(\Z_p,A_\pr)$
of the module generated by the functions
$\vartheta_u:\Z_p\rightarrow A_\pr$, $y \mapsto u^y$, for all $u\in U_1$.
(If $F_v$ is the completion of $F$ at a place $v$, the element $u\in F_v$ satisfies $|1-u|<1$ and $y\in\Z_p$, we put
\[ u^y:=\sum_{n\geq 0} {\binom{y}{n}}(u-1)^n\in F_v^* \ , \]
where ${\binom{y}{n}}$ is the reduction modulo $p$ of the value of the usual binomial.)

The next theorem follows from Sinnott's results and ideas in \cite{Si} applied to our setting.

\begin{thm}\label{teosmap} There is an injective morphism
\[ s\colon\Lambda/p\Lambda\hookrightarrow Dir(\Z_p,A_\pr) \]
such that for any $\g\in\G$, one has $s(\g)=\vartheta_{\kappa(\g)}$.
\end{thm}

\begin{proof} For any ring $R$, the algebra of $R$-valued distributions on $\G$ can be identified with $R[[\G]]$.
In \cite[Theorem 1]{Si} Sinnott constructs an isomorphism $A_\pr[[U_1]]\rightarrow Dir(\Z_p,A_\pr)$ by attaching to a measure $\mu$
the function $y\mapsto\int_{U_1} u^y d\mu(u)$. In particular, $\vartheta_u$ corresponds to the Dirac delta at $u$.
To complete our proof, one just has to recall that $\Lambda/p\Lambda=\F_\pr[[\G]]$ is a subring of $A_\pr[[\Gamma]]$ and compose
Sinnott's isomorphism with the one $A_\pr[[\G]]\simeq A_\pr[[U_1]]$ induced by the cyclotomic character $\kappa$.
\end{proof}

It is clear from the proof that the image of $s$ is exactly the closure of the $\F_\pr$-module generated by
the functions $\vartheta_u$.

\begin{prop} The morphism $s$ induces a ring homomorphism
\[ s_n\colon\F_\pr[\G_n]\rightarrow C^0(\Z_p,A/\pr^{n+1})\ .\]
\end{prop}

\begin{proof} It suffices to remark that $\k(\G^{(n)})=U_{n+1}$. Hence for $\g\in\G^{(n)}$ and $y\in\Z_p$ we have
\[ s(\g)(y)=\vartheta_{\k(\g)}(y)=\k(\g)^y\in 1+\pr^{n+1}A_\pr \]
which implies that the ideal $(\g-1\,:\,\g\in\G^{(n)})$ in $\F_\pr[[\G]]$ is sent by $s$ into $C^0(\Z_p,\pr^{n+1}A_\pr)$.
The kernel of the natural projection $\F_\pr[[\G]]\rightarrow\F_\pr[\G_n]$ is precisely the closure of this ideal (that is,
the image in $\F_\pr[[\Gamma]]$ of $\mathfrak{I}_n\subset \Lambda$).
\end{proof}

We get a commutative diagram
\[ \begin{CD} \F_\pr[[\G]] @>s>> Dir(\Z_p,A_\pr)\\
@VVV @VVV \\
\F_\pr[\G_n] @>s_n>> C^0(\Z_p,A/\pr^{n+1})  \end{CD} \]
where the vertical maps are the natural ones. By construction one has $s=\plim{n}s_n$.

\begin{prop}\label{snNotInj}
The map $s_n$ is not injective for $n>0$. \end{prop}

\begin{proof} Fix $n>0$ and choose $a_1$, $a_2$ and $a_3$ in $A_\pr$ so that $a_i$ and $a_j$ are different modulo
$\pr$ if $i\neq j$ (this is possible since we assume $q>2$). Consider the elements $\gamma_i\in\Gamma$ defined by
\[ \kappa(\g_i)=1+a_i\pi_\pr^n \qquad {\rm for }\ i=1,2,3 \ .\]
Our hypothesis on the $a_i$'s implies that the $\gamma_i$'s have $\F_\pr$-linearly independent images in $\F_\pr[\Gamma_n]$.
We need to find $x_i\in \F_\pr$ so to have
\[ x_1\kappa(\g_1)^y+x_2\kappa(\g_2)^y+x_3\kappa(\g_3)^y\in \pr^{n+1} A_\pr \]
for all $y\in\Z_p$. This is equivalent to
\[ x_1+x_2+x_3+y(x_1a_1+x_2a_2+x_3a_3)\pi_\pr^n \equiv 0 \pmod{\pr^{n+1}} \ ,\]
i.e.,
\[ \left\{\begin{array}{l}
x_1+x_2+x_3 = 0\\
x_1a_1+x_2a_2+x_3a_3 \equiv 0 \pmod{\pr} \end{array} \right. \ .\]
For any nontrivial solution of this linear system, the image in $\F_\pr[\Gamma_n]$ of $x_1\g_1+x_2\g_2+x_3\g_3$ is a
nontrivial element of the kernel of $s_n$.
\end{proof}

\subsection{Characters of $\Delta$}\label{SecTwoChar}
Let $\omega_\pr\colon A\rightarrow A_\pr$ be the morphism of $\F$-algebras obtained composing $A\twoheadrightarrow A/\pr=\F_\pr$ with
the lift $\F_\pr\hookrightarrow A_\pr$ (i.e., the {\em Teichm\"uller character} in positive
characteristic\,\footnote{The map $\omega_\pr$ can also be defined as the morphism of $\F$-algebras such that
$v_\pr(\theta-\omega_{\pr}(\theta))\geq1$: it satisfies $\omega_\pr(a)\equiv a \pmod{\pr}$
and corresponds to the choice of a root of $\pi_{\pr}$ in $\overline\F$ (because $\pi_\pr=\pi_\pr(\theta)\in A=\F[\theta]$ and we have
$\pi_\pr(\theta)\equiv \pi_\pr(\omega_{\pr}(\theta))\pmod\pr$, therefore $\pi_\pr(\omega_{\pr}(\theta))\equiv 0$).}).
Then any $a\in A-\pr$ is uniquely decomposed as
\begin{equation} \label{EqOmp} a=\omega_\pr(a)\langle a\rangle_{\pr} \end{equation}
where $\langle a\rangle_{\pr}\in U_1=1+\pr A_\pr$. The domain of $\omega_\pr$ can be extended to all of
$A_\pr$ (note that then the restriction of $\omega_\pr$ to $\F_\pr$ is just the identity) and equality \eqref{EqOmp}
holds for any $a\in A_\pr-\pr A_\pr$.

The restriction of $\omega_\pr\circ\kappa\colon G_\infty\rightarrow \F_\pr^*$ to $\Delta$ yields an isomorphism
$\Delta\rightarrow \F_\pr^*$, which in the rest of the paper will be simply denoted $\omega_\pr$, by an abuse
of notation meant to emphasize the ``Teichm\"uller-like'' quality of this characteristic $p$ character.
If, for any $a\in A-\pr$, we let $\sigma_a\in\Delta$ be the element such that
$\sigma_a(\varepsilon)=\Phi_a(\varepsilon)\;\forall\,\varepsilon\in\Phi[\pr]$ (recall that $\Delta\simeq\Gal(F_0/F)$), then we have
\[ \omega_\pr(\sigma_a)=\omega_{\pr}(a)\ . \]
Composition of $\kappa|_\Delta$ with the Teichm\"uller lift $\F_\pr^*\rightarrow\bmu_{q^d-1}$ yields a
character $\wt{\omega}_\pr\colon\Delta\rightarrow W^*$ (the {\em Teichm\"uller character} in characteristic 0). It satisfies
\[ \wt{\omega}_\pr(\sigma_a) \equiv \omega_\pr(a)\pmod{pW}\ . \]

A ($p$-adic) character $\chi$ on $\Delta$ is called {\em odd} if $\chi(\cali_\infty)\neq 1$ and {\em even} if $\chi(\cali_\infty)=1$.
Since all such characters are powers of $\wt\omega_\pr$, this definition amounts to say that $\wt\omega_\pr^i$ is even
if and only if $q-1$ divides $i$.

\subsubsection{Decomposition by characters}
For any $p$-adic character $\chi\in \Hom(\Delta,W^*)=:\widehat\Delta$ we put
\begin{equation}\label{DefIdemp}
e_\chi := \frac{1}{|\Delta|} \sum_{\delta\in \Delta} \chi(\delta^{-1})\delta \in W[\Delta] \end{equation}
for the idempotent associated with $\chi$ (we shall denote the trivial character by $\chi_0$).
We recall a few basic relations for the $e_\chi$:\begin{itemize}
\item for any $\delta\in \Delta$,
\begin{equation} \label{eqchidel} e_\chi\delta=\chi(\delta)e_\chi\,; \end{equation}
\item for any $\psi\in\wh\Delta$,
\[ \psi(e_\chi) =\left\{ \begin{array}{cl} 1 & {\rm if}\ \psi=\chi \\
0 & {\rm if}\ \psi\neq\chi \end{array} \right. \ ;\]
\item $\ds{\sum_{\chi\in\wh\Delta} e_\chi=1}$.
\end{itemize}
As usual, for any $W[\Delta]$-module $M$, we denote by $M(\chi)$ the $\chi$-part of $M$ (i.e., the submodule
$e_{\chi}M$) and we have a decomposition
\begin{equation} \label{eqMchi}
M \simeq \bigoplus_{\chi\in\widehat\Delta} M(\chi) \ .\end{equation}

\section{$\pr$-adic interpolation of the Carlitz-Goss $L$-function}\label{SecStick}
In this section we present the analytic side of our work, i.e., the Carlitz-Goss $\zeta$-function $\zeta_A$ and the
$\pr$-adic $L$-function we shall use to interpolate $\zeta_A$ at integers. Moreover we introduce the Stickelberger series which
will appear also in the computation of Fitting ideals of Tate modules and class groups in Section \ref{SecFitt}. Actually,
the Stickelberger series is the main hero of this section: as we shall see, it plays a universal role in interpolating
$L$-functions attached to abelian characters with no ramification outside a prescribed locus. In the case of $\C$-valued
characters and the complex $L$-functions attached to them, this will be clear from \eqref{eqls}. In Theorems \ref{thmtetzet}
and \ref{thmtetpr}, we shall see how, taking characteristic $p$-valued characters, the Stickelberger series interpolates
the $L$-functions defined by Goss. We also remark that in \cite{lltt2} the Stickelberger series is used as a $p$-adic $L$-function.

For the convenience of the reader we will recall different constructions and properties: some of them are known but we
lack an explicit reference including all of them.

\subsection{The Stickelberger series}
Recall that $\mathscr{P}_F$ is the set of places of $F$. Places different from
$\infty$ will be often identified with the correponding prime ideals of $A$.

The subset of $\mathscr{P}_F$ where the extension $\calf/F$ ramifies is $S=\{\pr,\infty\}$. Define $G_S$ as the Galois group
of the maximal abelian extension $\calf_S$ of $F$ which is unramified outside $S$. For any $\gotq\in\mathscr{P}_F-S$,
let $\Fr_\gotq\in G_S$ denote the corresponding (arithmetic) Frobenius automorphism.

\begin{defin}\label{DefStick1}
We define the {\em Stickelberger series} by
\begin{equation}\label{Sticel}
\Theta_{\calf_S/F,S}(X):= \prod_{\gotq\in\mathscr{P}_F-S}(1-\Fr_\gotq^{-1} X^{\deg(\gotq)})^{-1}\in \Z[G_S][[X]]\ .
\end{equation}
More generally, for any closed subgroup $U<G_S$, we define
\[ \begin{array}{ll} \Theta_{\calf_S^U/F,S}(X) & := \pi_{G_S/U}^{G_S}(\Theta_{\calf_S/F,S})(X) \\
\ & \\
\ & \ =\ds{\prod_{\gotq\in\mathscr{P}_F-S}(1-\pi_{G_S/U}^{G_S}(\Fr_\gotq^{-1})X^{\deg(\gotq)})^{-1} }
\in \Z[\Gal(\calf_S^U/F)][[X]]\ , \end{array} \]
where $\pi_{G_S/U}^{G_S}\colon\Z[G_S]\rightarrow\Z[\Gal(\calf_S^U/F)]$ is the map induced by the projection $G_S\twoheadrightarrow G_S/U$.
\end{defin}

The series in \eqref{Sticel} is well-defined, since for any $k$ there are only finitely many places of degree $k$.

\subsubsection{Convergence}
Let $R$ be a topological ring, complete with respect to a non-archimedean absolute value. The algebra $R[[G_S]]$
is the inverse limit of $R[\Gal(E/F)]$ as $E$ varies among finite subextensions of $\calf_S/F$; as such, it is endowed
with a topological structure. (To\-po\-lo\-gi\-cal\-ly each $R[\Gal(E/F)]$ is the product of $[E:F]$ copies of $R$
and $R[[G_S]]$ has the coarsest topology such that all projections $R[[G_S]]\rightarrow R[\Gal(E/F)]$ are continuous.)

For any topological ring $\calr$ the Tate algebra $\calr\langle X\rangle$ consists of those power series in $\calr[[X]]$
whose coefficients tend to 0. In particular, $R[[G_S]]\langle X\rangle$ contains all those power series whose image in
$R[G_S/U][[X]]$ is a polynomial for any open subgroup $U<G_S$.

For any unitary $R$, the natural map $\Z\rightarrow R$ (by $1\mapsto 1$) allows to think of $\Theta_{\calf_S/F,S}(X)$
as an element in $R[[G_S]][[X]]$. Moreover, for any group homomorphism $\alpha\colon G_S\rightarrow R^*$, the extension
by linearity to a map $\alpha\colon\Z[G_S]\rightarrow R$ yields a power series $\alpha\big(\Theta_{\calf_S/F,S}\big)(X)\in R[[X]]$.

\begin{prop} \label{prtatalg}
Let $R$ be a unitary topological $\Z_p$-algebra, complete with respect to a non-archimedean absolute value.
Then the series $\Theta_{\calf_S/F,S}(X)$ defines an element in the Tate al\-ge\-bra $R[[G_S]]\langle X\rangle$.
\end{prop}

\begin{proof} The proof is essentially the same as in \cite[Proposition 4.1.1]{lltt2} (see also \cite[\S5.3]{BBL1}),
so here we just sketch the basic ideas.

Let $\psi\colon G_S\rightarrow\C^*$ be a continuous character ($G_S$ has the profinite topology, so $\psi$ factors through
a subgroup of finite index). Then
\begin{equation} \label{eqls}
\psi\big(\Theta_{\calf_S/F,S}\big)(q^{-s})=
\prod_{\gotq\in\mathscr{P}_F-S}\left(1-\frac{\psi(\Fr_\gotq^{-1})}{(N\gotq)^s}\right)^{-1}=:L_S(s,\psi)
\end{equation}
is (possibly up to the Euler factors from places in $S$) the classical complex $L$-function attached to $\psi$
(here $N\gotq:=q^{\deg(\gotq)}$ is the order of the finite field $\F_\gotq$ and we assume $Re(s)>1$ to ensure
convergence of the infinite product). More precisely, one has
\begin{equation}\label{DefGlobLSer}
L(s,\psi):= L_S(s,\psi)\cdot\prod_{v\in S}(1-\psi(v)q^{-s\deg(v)})^{-1} \ , \end{equation}
here $\psi(v)$ denotes the value of $\psi$ on the inverse of the Frobenius of $v$ (an element in $G_S/\ker(\psi)$
if $v$ is not ramified in $\calf_S^{\ker(\psi)}/F$), with the usual convention $\psi(v)=0$ if $\calf_S^{\ker(\psi)}/F$ is ramified at $v$.

A theorem of Weil (see, e.g., \cite[V, Th\'eor\`eme 2.5]{Ta}) implies that $L(s,\psi)$ is a polynomial in $q^{-s}$,
unless $\psi=\psi_0$ is trivial, then one has
\begin{equation}\label{DefGlobLSerTriv} L(s,\psi_0)=\frac{1}{(1-q^{-s})(1-q^{1-s})} \ . \end{equation}
Thus $L_S(s,\psi)$ is a rational function of $q^{-s}$, with denominator bounded independently of $\psi$.

Choose an auxiliary place $\gotq_0\notin S$ and put
\[ \Theta_{\calf_S/F,S,\{\gotq_0\}}(X):=
(1-q^{\deg(\gotq_0)}\Fr_{\gotq_0}^{-1}X^{\deg(\gotq_0)})\Theta_{\calf_S/F,S}(X) \ .\]
By Weil's theorem for all $\psi$ as above, $\psi\big(\Theta_{\calf_S/F,S,\{\gotq_0\}}\big)(q^{-s})$ belongs to $\C[q^{-s}]$
(more precisely, to $\Z[\psi(G_S)][q^{-s}]$\,). As a consequence, one gets
\[ \pi_{G_S/U}^{G_S}(\Theta_{\calf_S/F,S,\{\gotq_0\}})(X)\in\Z[\Gal(\calf_S^U/F)][X] \]
for all open subgroups $U<G_S$ and hence $\Theta_{\calf_S/F,S,\{\gotq_0\}}(X)\in\Z[[G_S]]\langle X\rangle$. It follows
that also $\Theta_{\calf_S/F,S}(X)$ is in $\Z_p[[G_S]]\langle X\rangle$, because the ratio between
$\Theta_{\calf_S/F,S,\{\gotq_0\}}(X)$ and $\Theta_{\calf_S/F,S}(X)$ is a unit in the Tate algebra.

Finally, for $R$ as in the hypothesis, the natural map $\Z_p\rightarrow R$ is extended to a continuous homomorphism
$\Z_p[[G_S]][[X]]\rightarrow R[[G_S]][[X]]$. Our proposition follows from the restriction
$\Z_p[[G_S]]\langle X\rangle\rightarrow R[[G_S]]\langle X\rangle$.
\end{proof}

\begin{rem} We remind readers of two important properties of Stickelberges series. Let $E/F$ be a finite subextension of
$\calf_S/F$. Then one has \begin{enumerate}
\item  $\Theta_{E/F,S}(X)\in\frac{1}{(1-q X)|G|}\ov{\Z}[G][X]\cap\Z[G][[X]]$ (where $\ov{\Z}$ is the integral closure of
$\Z$ in the algebraic closure of $\Q$), see \cite[Theorem 15.13]{Ro};
\item the Brumer-Stark element $w_{E/F}=(q-1)\Theta_{E/F,S}(1)\in\Z[G]$ annihilates $\CaCl(E)$, see \cite[Chapter 15]{Ro}.
\end{enumerate}
\end{rem}

\begin{thm} \label{thmconv} Let $R$ be as in Proposition \ref{prtatalg} and $\alpha\colon G_S\rightarrow R^*$ a continuous
group homomorphism. The power series $\alpha\big(\Theta_{\calf_S/F,S}\big)(X)$ converges on the unit disk $\{x\in R:|x|\leq1\}$.
\end{thm}

\begin{proof} The ring homomorphism $R[[G_S]]\rightarrow R$ induced by $\alpha$ is continuous; hence it extends to a
homomorphism of Tate algebras $R[[G_S]]\langle X\rangle\rightarrow R\langle X\rangle$. Thus, by Proposition
\ref{prtatalg}, $\alpha\big(\Theta_{\calf_S/F,S}\big)(X)$ is in $R\langle X\rangle$, which, by definition, consists
exactly of those power series convergent on the unit disk.
\end{proof}

\subsubsection{Stickelberger series in the $\pr$-cyclotomic tower} \label{sssStickcycl}
In the following, we shall be particularly interested in the image of the Stickelberger series along the
$\pr$-cyclotomic tower. Define
\begin{equation}\label{globStickinf}
\Theta_\infty(X):=\Theta_{\calf/F,S}(X)\in\Z[G_\infty][[X]] \end{equation}
and, for all $n\in\N$,
\[ \Theta_n(X):=\Theta_{F_n/F,S}(X)\in\Z[G_n][[X]] \ .\]
We shall think of $\Theta_\infty$ and $\Theta_n$ as power series with coefficients respectively in $W[[G_\infty]]$ and $W[G_n]$.

Any element in $G_\infty$ can be uniquely written as $\delta\gamma$, with $\delta\in\Delta$ and $\gamma\in\Gamma$.
Consequently, given $\chi\in \Hom(\Delta,W^*)$ we can define a group homomorphism $G_\infty\rightarrow\Lambda^*$
by $\delta\gamma\mapsto\chi(\delta)\gamma$. By linearity and continuity, this can be extended to a ring homomorphism
(which, by abuse of notation, we still denote by the same symbol) $\chi\colon W[[G_\infty]]\rightarrow\Lambda$.
The decomposition \eqref{eqMchi} applied to $W[[G_\infty]][[X]]$ then yields the following definition.

\begin{defin}\label{DefSticFuncCyc} For any $\chi\in \Hom(\Delta,W^*)$, the {\em $\chi$-Stickelberger series}
for the $\pr$-cyclo\-to\-mic tower is
\[ \Theta_{\infty}(X,\chi):=\chi(\Theta_{\infty})(X)\in\Lambda[[X]] \ .\]
Similarly, we put
\begin{equation} \label{tetn}
\Theta_n(X,\chi):=\chi(\Theta_n)(X) \in W[\G_n][[X]]\ . \end{equation}
\end{defin}

The series $\Theta_n(X,\chi)$ form a projective system: let $\pi^{n+1}_n\colon W[\G_{n+1}][[X]]\rightarrow W[\G_n][[X]]$
be the projection induced by the natural map $\G_{n+1}\twoheadrightarrow\G_n$, then we have
\[ \pi^{n+1}_n(\Theta_{n+1}(X,\chi))=\Theta_n(X,\chi) \]
and $\Theta_{\infty}(X,\chi)=\varprojlim\Theta_n(X,\chi)$ for all $\chi\in \wh{\Delta}$. Moreover, \eqref{eqchidel} yields
\begin{equation}\label{EqChiTheta} e_\chi\Theta_\infty(X)=\Theta_\infty(X,\chi)e_{\chi} \end{equation}
and
\[ \Theta_\infty(X)=\sum_{\chi\in\wh\Delta}\Theta_{\infty}(X,\chi)e_\chi \]
(of course these relations descend to level $n$ for all $n\in\N$). Finally, the proof of Proposition \ref{prtatalg}
shows that $\Theta_n(X,\chi)\in W[\G_n][X]$ if $\chi\neq\chi_0$ and $\Theta_n(X,\chi_0)\in\frac{1}{1-q X}W[\G_n][X]$.

\subsection{Carlitz-Goss $\zeta$-function and Bernoulli-Goss numbers}\label{CGBG}
We recall the con\-struc\-tion of Goss $L$-function and the main properties needed in our work (a
general re\-fe\-ren\-ce is \cite[Chapter 8]{Goss}).

As usual, $F_\infty$ denotes the completion of $F$ at $\infty$ and $\C_\infty$ is the completion of an algebraic closure
of $F_\infty$. The valuation on $F_\infty$ extends to $v_{\infty}\colon\C_{\infty}\rightarrow\Q\cup\{\infty\}$.
We also fix an embedding of $\ov{F}$ in $\C_\infty$. Finally, let $U_1(\infty)$ denote the group of 1-units in $F_\infty^*$.

Since we are taking $F=\F(\theta)$, a somewhat natural choice of uniformizer at $\infty$ is $\theta^{-1}$. Fixing a uniformizer
establishes a sign function $sgn\colon F_\infty^*\rightarrow\F^*$, which sends $x\in F_\infty^*$ into the residue of
$x\theta^{v_\infty(x)}$, and a projection
\[ F_\infty^*\longrightarrow U_1(\infty)\;,\quad \;x\mapsto\langle x\rangle_{\infty}:=\frac{x\theta^{v_\infty(x)}}{sgn(x)} \ .\]
Note that one has $\ker(sgn)=\theta^\Z\times U_1(\infty)$.

\begin{rem} These maps can be made more ``concrete'' by the following observation. Let $a\in A-\{0\}$ and write it as
$a=a_0+\ldots+a_n\theta^n$, with $n=\deg(a)$ and $a_i\in\F$.
Then we have
\begin{equation} \label{eqsgninft}
sgn(a)=a_n\in\F^*\quad\text{and} \quad  \langle a\rangle_{\infty}=
\frac{a}{\theta^{\deg(a)} sgn(a)} \in1+\theta\,\F[[\theta^{-1}]] \ .\end{equation}
\end{rem}

\subsubsection{The group $\bbS_\infty$} \label{sss:Sinf}
Let $\I_F$ denote the group of ideles of $F$. Then we have
\begin{equation} \label{eqidel}
\I_F/F^*\simeq\ker(sgn)\times\prod_{\gotq\in\mathscr{P}_F-\{\infty\}}A_\gotq^*=:\mathcal D \ , \end{equation}
where $A_\gotq$ denotes the completion of $A$ with respect to $\gotq$ and the isomorphism is given by the embedding
of the right-hand side as a subgroup of $\I_F$.

The group of $\C_\infty$-valued principal quasi-characters on $\I_F/F^*$ is
\[ \bbS_\infty:=\C_\infty^*\times\Z_p \ .\]
For $s=(x,y)\in \bbS_{\infty}$, we define a continuous homomorphism $\ker(sgn)\longrightarrow\C_\infty^*$ by
\begin{equation} \label{eqsinfexp}
a\mapsto a^s:=x^{-v_\infty(a)}\langle a\rangle_{\infty}^y \ .\end{equation}
This map is extended to all of $\I_F$ by the projection to $\ker(sgn)$ induced by the isomorphism \eqref{eqidel}.

The group structure on $\bbS_\infty$ is given by $(x_1,y_1)+(x_2,y_2):=(x_1x_2,y_1+y_2)$. We have an injection
$\Z\hookrightarrow \bbS_{\infty}$, by
\[ j\mapsto c_j:=(\theta^j,j) \ .\]
By \eqref{eqsgninft} we get $a^{c_j}=a^j$ for all $j\in\Z$ and monic $a\in A$.

In analogy with the complex half-plane $\C^+:=\{z\in\C\mid\Re(z)>1\}$, we define a ``half-plane''
\[ \bbS_\infty^+:=\{(x,y)\in\bbS_\infty:|x|>1\} \ . \]

\subsubsection{From $\Theta_{\calf_S/F,S}$ to $\zeta_A$}
Let $A_+$ be the set of monic polynomials in $A$.
Thinking of $A$ as a subset of $F_\infty$, we have $A_+=A\cap\ker(sgn)$.

\begin{defin}\label{DefGCZetaFunc}
The {\em Carlitz-Goss $\zeta$-function} is defined as
\begin{equation} \label{eqzetgoss} \zeta_A(s):=\sum_{a\in A_+}a^{-s} \quad ,\ s\in\bbS_\infty\ .\end{equation}
\end{defin}

\noindent For $s=(x,y)$, we have $a^s=x^{\deg(a)}\langle a\rangle_{\infty}^y$, hence $|a^{-s}|=|x|^{-\deg(a)}$.
It follows that the series \eqref{eqzetgoss} converges on $\bbS_\infty^+$. (Note the analogy with convergence of
the series defining $L(s,\psi)$ for $Re(s)>1$.)

Class field theory identifies the group $U_1(\infty)$ with a factor of $G_S$. Consequently, the construction in
Section \ref{sss:Sinf} can be used to define $\C_\infty$-valued characters on $G_S$. More precisely, for $y\in\Z_p$
let $\psi_y\colon G_S\rightarrow\C_\infty^*$ be the homomorphism obtained by composing the class field theoretic
projection $\rho\colon G_S\twoheadrightarrow U_1(\infty)$ with $(1,y)\in\bbS_\infty$. Then Theorem \ref{thmconv}
shows that $\psi_y(\Theta_{\calf_S/F,S})(x)$ converges for all $x\in\C_\infty$ such that $|x|\leq1$.

\begin{thm} \label{thmtetzet} For all $s=(x,y)\in\bbS_\infty^+$, we have
\begin{equation} \label{eqtetzet}
\psi_{-y}(\Theta_{\calf_S/F,S})(x^{-1})=(1-\pi_\pr^{-s})\zeta_A(s)\ . \end{equation}
\end{thm}

\noindent (Recall that $\pi_\pr$ is the monic irreducible generator of the ideal $\pr$ in $A$.)

\begin{proof} For every place $\gotq$ in $\mathscr{P}_F-S$, let $\pi_\gotq\in A_+$ denote the monic generator of the corresponding
prime ideal in $A$. Then \eqref{eqzetgoss} can be rewritten as an Euler product
\begin{equation} \label{eqeulprodzet}
\zeta_A(s)=\prod_{\gotq\in\mathscr{P}_F-\{\infty\}}(1-\pi_\gotq^{-s})^{-1}=
\prod_{\gotq\in\mathscr{P}_F-\{\infty\}}(1-\langle\pi_\gotq\rangle_\infty^{-y}x^{-\deg(\gotq)})^{-1}\ .
\end{equation}

By class field theory, we have a reciprocity map $rec\colon\I_F\rightarrow G_S$ with dense image isomorphic to $\ker(sgn)\times A_\pr^*$.
Using \eqref{eqidel}, the composition $\rho\circ rec$ is just the projection
\[ \mathcal D=\theta^\Z\times U_1(\infty)\times\prod_{\gotq\in\mathscr{P}_F-\infty}A_\gotq^*\longrightarrow U_1(\infty)\ .\]
For $\gotq$ in $\mathscr{P}_F-S$, let $i_\gotq\in\I_F$ denote the idele having $\pi_\gotq$ as its $\gotq$-component and 1 as
component at all other places: then $\Fr_\gotq=rec(i_\gotq)$. By the diagonal embedding $F^*\hookrightarrow\I_F$, we also get
$rec(i_\gotq)=rec(i_\gotq a)$ for all $a\in F^*$. Since $\pi_\gotq$ belongs to $F^*$ and $i_\gotq\pi_\gotq^{-1}$ is in the
fundamental domain $\mathcal D$, we finally obtain $\rho(\Fr_\gotq)=\langle\pi_\gotq^{-1}\rangle_\infty$\,.

Thus $\psi_{-y}(\Fr_\gotq^{-1})=\langle\pi_\gotq\rangle_\infty^{-y}$ and
\[ \psi_{-y}(\Theta_{\calf_S/F,S})(x^{-1})=
\prod_{\gotq\in\mathscr{P}_F-S}(1-\langle\pi_\gotq\rangle_\infty^{-y}x^{-\deg(\gotq)})^{-1}\ .\]
Comparison with \eqref{eqeulprodzet} completes the proof.
\end{proof}

\noindent Theorem \ref{thmtetzet} can be used to obtain analytical continuation of $\zeta_A$ on the ``boundary'' of
$\bbS_\infty^+$ (that is, $\{s=(x,y):|x|=1\}$), since, by Theorem \ref{thmconv}, the left-hand side of \eqref{eqtetzet}
converges if $|x|=1$.

\begin{rem} \label{remintegr} Let $R$ be a topological ring: then the ring of $R$-valued distributions\,\footnote{By $R$-valued
distributions on a locally profinite group $G$ we mean the
linear functionals on the space of compactly supported locally constant functions $G\rightarrow R$.}
on $\ker(sgn)$ is isomorphic to $R[[U_1(\infty)]][[X]]$. This suggests that
equation \eqref{eqtetzet} can be interpreted as providing an integral formula for the Carlitz-Goss zeta function (namely,
integration of the quasi-character $s$ against the distribution induced by the Stickelberger series; a variant of
this will be made explicit in the proof of Theorem \ref{StickLfun}). Integral formulae for $\zeta_A$ and its generalizations
were already known (starting with Goss's foundational paper \cite{goss2}; see \cite[\S5.7]{thakur} for a quick introduction
to the topic), but (to the best of our knowledge) were all based on measures on some additive group; our approach instead
stresses the role of the multiplicative group $\ker(sgn)$ and thus might provide some useful new insight. \end{rem}

\subsubsection{Bernoulli-Goss numbers}
Our final goal in this chapter is to interpolate the Carlitz-Goss zeta function at negative integers.
Lacking a functional equation, we have to use more brutal techniques in order to extend the domain of $\zeta_A$ to all
of $\bbS_\infty$.

For any $n\geq 1$ let $A_{+,n}:=\{a\in A_+\,:\,\deg(a)=n\}$ (note that $|A_{+,n}|=q^n$).
For any $j\in\Z$ and $n\in\N$ put
\[ S_n(j):=\sum_{a\in A_{+,n}}a^j \ .\]
Note that we have $S_0(j)=1$ for all $j\in\Z$.

\begin{lem}\label{Simon}
If $1\leq j<q^n-1$, then $S_n(j)=0$.
\end{lem}

\begin{proof}
See \cite[Remark 8.12.1.\,1]{Goss}.
\end{proof}

Reorganizing the terms in \eqref{eqzetgoss} we can also write the Carlitz-Goss $\zeta$-function as
\[ \zeta_A(x,y):= \sum_{n\geq 0}\left( \sum_{a\in A_{+,n}}
\langle a\rangle_{\infty}^{-y}\right)\, x^{-n} \quad ,\ (x,y)\in \bbS_{\infty} \ .\]
This second formula guarantees the convergence for all $s\in \bbS_{\infty}$ because of the following

\begin{lem} \label{ConvZeta}
For any $y\in \Z_p$ and any $n\geq 1$, one has
\[ v_{\infty}\left( \sum_{a\in A_{+,n}} \langle a\rangle_{\infty}^y\right )\geq p^{n-1} \ . \]
\end{lem}

\begin{proof} For $n=1$ there is nothing to prove, so fix $n\geq 2$. If $y=p^ny'$, we can write
$\langle a\rangle_{\infty}^{y'}=1+\tilde a$ where $v_\infty(\tilde a)\geq1$. The claim in this case is then obvious by
\[  \sum_{a\in A_{+,n}}\langle a\rangle_\infty^y= \sum_{a\in A_{+,n}}\langle a\rangle_\infty^{y'p^n}=
\sum_{a\in A_{+,n}}(1+\tilde a)^{p^n} = \sum_{a\in A_{+,n}}(1+\tilde a^{p^n})= \sum_{a\in A_{+,n}}\tilde a^{p^n} \]
(since we are in characteristic $p$ and $|A_{n,+}|=q^n$).

\noindent If $y\not\equiv 0\pmod{p^n}$, then take an integer $y_{n-1}\equiv y\pmod{p^{n-1}}$ with $1\leq y_{n-1} \leq p^{n-1}-1$.
Since $q\geq p$, we get $y_{n-1}<q^n-1$ and Lemma \ref{Simon} implies $S_n(y_{n-1})=0$. Therefore
\[ \sum_{a\in A_{+,n}} \langle a\rangle_\infty^{y_{n-1}} = \sum_{a\in A_{+,n}} \left(\frac{a}{\theta^n}\right)^{y_{n-1}}
=\frac{1}{\theta^{ny_{n-1}}} S_n(y_{n-1}) = 0\ .\]
Moreover
\[ \langle a\rangle_\infty^y-\langle a\rangle_\infty^{y_{n-1}} =
\langle a\rangle_\infty^{y_{n-1}}\left( \langle a\rangle_\infty^{p^{n-1}y'} -1 \right)=
\langle a\rangle_\infty^{y_{n-1}}\left( \langle a\rangle_\infty^{y'} -1 \right)^{p^{n-1}}
=\langle a\rangle_\infty^{y_{n-1}} c^{p^{n-1}} \]
(where $v_\infty(c)\geq 1$), so that
\[ v_\infty(\langle a\rangle_\infty^y-\langle a\rangle_\infty^{y_{n-1}}) \geq p^{n-1} \ .\]
Hence
\[ v_{\infty}\left( \sum_{a\in A_{+,n}} \langle a\rangle_{\infty}^y\right ) =
v_{\infty}\left( \sum_{a\in A_{+,n}} \langle a\rangle_{\infty}^y - \langle a\rangle_\infty^{y_{n-1}} \right)
\geq p^{n-1} \ .  \qedhere\]
\end{proof}

For any $j\in\N$ and $x\in \C_\infty^*$, we have the equality
\begin{equation}\label{Zetaj}
\zeta_A\left(\frac{x}{\theta^j},-j\right)=\sum_{a\in A_+} x^{-\deg(a)}\theta^{j\deg(a)}
\left(\frac{a}{\theta^{\deg(a)}}\right)^j = \sum_{n\geq 0} S_n(j)x^{-n} \ ,
\end{equation}
which leads to the following

\begin{defin}\label{DefZXjBGPol}\
\begin{enumerate}
\item For any integer $j\geq 0$ we put
\begin{equation}\label{eqZetaxy} Z(X,j):=\sum_{n\geq 0} S_n(j)X^n\in A[X] \end{equation}
(it is a polynomial because of Lemma \ref{Simon}).
\item For any $j\in\N$, the {\em Bernoulli-Goss numbers} $\beta(j)$ are defined as
\[ \beta(j):=\left\{\begin{array}{cl} Z(1,j) & \textrm{if} \ j=0\ \textrm{or}\ j\not\equiv0\pmod{q-1} \\
\ & \\
-\frac{d}{dX}Z(X,j)|_{X=1}& \textrm{if}\ j\geq 1\ \textrm{and}\ j\equiv 0\pmod{q-1} \end{array}\right. \ .\]
\end{enumerate}
\end{defin}

\noindent By definition, for any $j\in\N$, we have
\[ \zeta_A(-j)=Z(1,j) \]
(by an abuse of notation, we write $\zeta(-j)$ for $\zeta(c_{-j})$). It is known that, for $j\geq 1$ with
$j\equiv 0 \pmod{q-1}$, we have $Z(1,j)=0$, which corresponds to a trivial zero in this setting (see \cite[Example 8.13.6]{Goss}).
Moreover it is clear that $\beta(j)\in A$ and $\beta(j)=1$ for $0\leq j\leq q-2$. We also have $\beta(q-1)=1$, as can be
deduced from the following lemma.

\begin{lem}\label{BetaCongrTheta}
For all $j\geq 0$, we have $\beta(j)\equiv 1 \!\pmod{\theta^q-\theta}$. In particular $\beta(j)\neq 0$.
\end{lem}

\begin{proof}
Recall that $\beta(0)=1$ and $S_0(j)=1$ for any $j\geq 0$. For any $\alpha\in\F$ we can write a polynomial $a\in A_{+,n}$
in terms of powers of $\theta-\alpha$, i.e., $a=a_0+a_1(\theta-\alpha)+\cdots+(\theta-\alpha)^n$. Therefore, for any $j\geq 1$,
\[ S_n(j)=\sum_{a\in A_{+,n}}a^j\equiv q^{n-1}\sum_{a_0\in\F} a_0^j \pmod{\theta-\alpha} \ .\]
Thus $S_n(j)\equiv 0\pmod{\theta-\alpha}$ for any $n\geq 2$ or for $n=1$ and
$j\not\equiv 0\pmod{q-1}$. Moreover for $n=1$ and $j\equiv 0\pmod{q-1}$ one has
\[ S_1(j)\equiv \sum_{a_0\in\F} 1 \equiv -1 \pmod{\theta-\alpha} \ .\]
Hence
\[ Z(X,j) \equiv \left\{\begin{array}{cl}
S_0(j)-X\equiv 1-X \!\!\pmod{\theta-\alpha}& {\rm if}\ j\geq 1 \ {\rm and} \ j\equiv 0\!\!\pmod{q-1}\\
\ & \\
S_0(j)\equiv 1 \!\!\pmod{\theta-\alpha} & {\rm otherwise}
\end{array} \right. \ .\]
The lemma follows by the definition of $\beta(j)$ (recalling that the terms $\theta-\alpha$ are relatively
prime and their product is $\theta^q-\theta$).
\end{proof}

\subsection{$\pr$-adic $L$-function and interpolation}\label{SecprLfun}
The previous section dealt with the prime at infinity, now we focus on the other place in $S$. We give here
the details of the construction of Goss's $\pr$-adic $L$-function (see \cite{Goss} for more).

\subsubsection{The group $\bbS_\pr$} %\label{sss:Spr}
Similarly to $\bbS_\infty$, we define a group of $\C_\pr$-valued quasi-characters on $\I_F/F^*$ by
\[ \bbS_\pr:=\C_\pr^*\times\Z_p\times\Z/|\F_\pr^*| \ .\]
However, in this case we shall be interested only in characters factoring through the compact group $A_\pr^*$.
So we embed $\Z$ into $\bbS_\pr$ by $j\mapsto(1,j,j)$. (Note that the image of this map is dense in
$\{1\}\times\Z_p\times\Z/(q^d-1)$, in contrast with the discrete embedding $\Z\hookrightarrow\bbS_\infty$\,.
This should be compared with the fact that $\Z$ is discrete in $\C$, but not in the $p$-adics.)
\footnote{This definition - the same as in \cite[\S5.5(b)]{thakur} - differs from the one in \cite{Goss},
where the factor $\C_\pr^*$ is missing. We decided to insert this factor in order to emphasize the
symmetry with $\bbS_\infty$.}

For $s=(1,y,i)\in \bbS_\pr$ and $a\in A_\pr^*$, the decomposition \eqref{EqOmp} suggests to define
\[ a^s:=\omega_\pr^i(a)\langle a\rangle_\pr^y \ . \]
Then we obtain a continuous homomorphism $\xi_s\colon G_S\rightarrow\C_\pr^*$ as composition of the maps
\begin{equation} \label{eqxis}
\begin{CD} G_S @>{\sigma\mapsto\sigma|_\calf}>> G_\infty @>{\kappa}>> A_\pr^* @>{a\mapsto a^s}>> \C_\pr^* \ . \end{CD}
\end{equation}

\subsubsection{The $\pr$-adic $L$-function}
As with $\zeta_A$, we first define a function by a certain power series and then interpret it as specialization
of the Stickelberger series.

\begin{defin}\label{DefpLFunc}
For any $0\leq i\leq q^d-2$ and any $y\in\Z_p$, we define the {\em $\pr$-adic $L$-function} as
\begin{equation}\label{eqpadicLfun} L_\pr(X,y,\omega_\pr^i):=
\sum_{n\geq0}\left(\sum_{a\in A_{+,n}-\pr} \omega_\pr^i(a)\langle a\rangle_\pr^y \right) X^n \ . \end{equation}
\end{defin}

\noindent Note that $L_\pr(X,y,\omega_\pr^i)$ is a element of $A_{\pr}[[X]]$: as such, it converges on the
open unit disc of $\C_\pr$. We can think of it as a function defined on $\bbS_\pr^+:=\{(x,y,i)\in\bbS_\pr:|x|<1\}$.

\begin{thm} \label{thmtetpr} We have
\begin{equation} \label{eqtetpr}
\xi_{-s}(\Theta_{\calf_S/F,S})(X)=L_\pr(X,y,\omega_\pr^i) \end{equation}
for every $s=(y,i)\in\Z_p\times\Z/(q^d-1)$.
\end{thm}

\begin{proof}
Equation \eqref{eqpadicLfun} can be rewritten as an Euler product
\begin{equation}\label{EulProdLFunc}
L_{\pr}(X,y,\omega_\pr^i)=
\prod_{\gotq\in\mathscr{P}_F-S}\big(1-\omega_\pr^i(\pi_\gotq)\langle\pi_\gotq\rangle_\pr^y X^{\deg(\gotq)}\big)^{-1} \ .
\end{equation}
Thus, as in the proof of Theorem \ref{thmtetzet}, we just need to check that the equality
\[ \xi_{-s}(\Fr_\gotq^{-1}) =
\omega_\pr^i(\pi_\gotq)\langle\pi_\gotq\rangle_\pr^y=\pi_\gotq^s \]
holds for every $s$ and $\gotq$. An element in $G_\infty$ is completely determined by its action on $\Phi[\pr^\infty]$;
since $\Phi_{\pi_\gotq}(x)\in A[x]$ is monic and it satisfies
\[ \Phi_{\pi_\gotq}(\varepsilon)\equiv \varepsilon^{\deg(\gotq)}\pmod\gotq \]
for every $\varepsilon\in\Phi[\pr^\infty]$, we get $\Phi_{\pi_\gotq}(\varepsilon)=\Fr_\gotq(\varepsilon)$.
Then \eqref{eqkappa} implies that the restriction of $\Fr_\gotq$ to $\calf$ is exactly $\kappa^{-1}(\pi_\gotq)$.
\end{proof}

Theorem \ref{thmtetpr} implies that the series $L_\pr(X,y,\omega_\pr^i)$ converges on the closed unit disc.
Actually, one can show that \eqref{eqpadicLfun} defines an entire function on $\C_\pr$, by a reasoning similar
to the one of Lemma \ref{ConvZeta}. Since we are only interested in the specialization at $X=1$, we won't discuss
the matter any further (see \cite[Chapter 8]{Goss} for more).

\begin{cor}\label{cor2padicLfun} Let $j$ be a natural number congruent to $i\pmod{q^d-1}$. Then
\begin{equation} \label{eqLprZ} L_{\pr}(X,j,\omega_\pr^i)=(1-\pi_\pr^jX^d)Z(X,j)\in A[X] \end{equation}
and, for any $y\in\Z_p$, we have
\begin{equation} \label{eqLprZ2} L_{\pr}(X,y,\omega_\pr^i)\equiv Z(X,i) \pmod{\pr} \ .\end{equation}
\end{cor}

\begin{proof} Since $j$ is an integer and $i$ is its reduction modulo $q^d-1$, we have
\[ \xi_{(-j,-i)}(\Fr_\gotq^{-1})
=\omega_\pr^i(\pi_\gotq)\langle\pi_\gotq\rangle_\pr^j=\pi_\gotq^j=
\langle\pi_\gotq\rangle_\infty^j\cdot\theta^j = \psi_j(\Fr_\gotq^{-1})\cdot\theta^j \]
for all places $\gotq\notin S$. Therefore Theorem \ref{thmtetpr} gives an equality of power series in $F[[X]]$
\[ L_{\pr}(X,j,\omega_\pr^i) = \xi_{(-j,-i)}(\Theta_{\calf_S/F,S})(X) =
\psi_j(\Theta_{\calf_S/F,S})(\theta^jX)\ .\]
It is convenient to extend the exponentiation in \eqref{eqsinfexp} by $a^{(xX,y)}:=\langle a\rangle_\infty^y(xX)^{-v_\infty(a)}$
(where $x,y$ are as in \eqref{eqsinfexp} and $X$ is a formal variable). Then Theorem \ref{thmtetzet} yields
\[ \psi_j(\Theta_{\calf_S/F,S})(\theta^jX)=
(1-\pi_\pr^{-(\theta^jX,j)})\cdot\zeta_A\left(\frac{1}{\theta^jX},-j\right)=(1-\pi_\pr^jX^d)\, Z(X,j) \]
(the first equality is just a restatement of \eqref{eqtetzet} in terms of Laurent series and the second one follows
from \eqref{Zetaj}).

\noindent As for \eqref{eqLprZ2}, it is enough to observe that one has $\langle a\rangle_\pr^y\equiv1\pmod{\pr}$ for any
$a\in A_\pr^*$ and $y\in\Z_p$. Hence \eqref{eqpadicLfun} shows that the variable $y$ is irrelevant modulo $\pr$
and \eqref{eqLprZ} yields
\[ L_{\pr}(X,y,\omega_\pr^i) \equiv L_\pr(X,i,\omega_\pr^i) \equiv Z(X,i) \pmod{\pr}\ .\qedhere\]
\end{proof}

\begin{rem}
A more direct proof of \eqref{eqLprZ} can be obtained from the equation
\[ L_{\pr}(X,j,\omega_\pr^i)=\sum_{n\geq 0}(S_n(j)-\pi_\pr^j S_{n-d}(j))X^n \ ,\]
which is obvious from \eqref{eqpadicLfun}.
However, the devious path we followed might be forgiven considering that it illustrates how \eqref{eqLprZ} and
\eqref{eqtetzet} are essentially the same statement.
\end{rem}

Regarding the special values of $L_\pr(X,i,\omega_\pr^i)$ we have the following

\begin{lem}\label{lem4padicLfun} If $i\equiv 0\pmod{q-1}$, then $L_{\pr}(1,y,\omega_\pr^i)=0$ for all $y\in \Z_p$. \end{lem}

\begin{proof} For any $y\in\Z_p$ and $m\geq 1$, take $j\in\N-\{0\}$ such that $j\equiv i\pmod{q^d-1}$ and $j\equiv y\pmod{p^m}$.
Then one has
\[ \langle a\rangle_\pr^y\equiv\langle a\rangle_\pr^{y-j}\langle a\rangle_\pr^j\equiv\langle a\rangle_\pr^j\pmod{\pr^{p^m}} \]
for any $a\in A_\pr^*$ and hence, by \eqref{eqpadicLfun} and \eqref{eqLprZ},
\begin{equation} \label{eqcongrprm}
L_\pr(X,y,\omega_\pr^i) \equiv  L_\pr(X,j,\omega_\pr^i) = (1-\pi_\pr^jX^d)Z(X,j) \pmod{\pr^{p^m}} \ .  \end{equation}
Now, since $j\geq 1$ and $j\equiv i\equiv 0\pmod{q-1}$, we have $Z(1,j)=0$ (see \cite[Example 8.13.6]{Goss}) and the lemma
follows taking the limit as $m$ goes to infinity.
\end{proof}

\noindent We also recall one of the main results of \cite{AnTa}.

\begin{thm}\cite[Theorem E]{AnTa}
Let $0\leq i\leq q^d-2$ with $i\not\equiv 0\pmod{q-1}$. Then
\[ L_{\pr}(1,-1,\omega_\pr^i)\neq 0 \ .\]
\end{thm}

\begin{rem} It would be interesting to investigate further the values of $L_{\pr}(1,y,\omega_\pr^i)$ for odd $i$.
From equation \eqref{eqLprZ2} one immediately has that
\[ Z(1,i) \not\equiv 0 \pmod{\pr} \Longrightarrow L_\pr(1,y,\omega_\pr^i) \neq 0\ \ \forall y\in\Z_p \ .\]
In general: is it true that for any $0\leq i\leq q^d-2$ with $i\not\equiv 0\pmod{q-1}$ and for any $y\in\Z_p$,
we have $L_\pr(1,y,\omega_\pr^i)\neq 0$ ?
\end{rem}

We end this section by providing another formula for $L_{\pr}(X,y,\omega_\pr^i)$. The Sinnott map $s$ of Theorem \ref{teosmap}
induces a map
\[ s_X\colon\Lambda[[X]]\rightarrow Dir(\Z_p,A_{\pr})[[X]] \]
in the obvious way, sending $\sum_nc_nX^n\in\Lambda[[X]]$ into the function $y\mapsto\sum_ns(\bar c_n)(y)X^n$
(where $\bar c_n$ is the reduction of $c_n$ modulo $p$).

\begin{thm}\label{StickLfun} For every $y\in\Z_p$ and $i\in\Z/(q^d-1)\Z$, we have
\begin{equation} \label{eqsinnthet}
s_X(\Theta_\infty(X,\wt\omega_\pr^{-i}))(y)=L_\pr(X,-y,\omega_\pr^i) \ .\end{equation}
\end{thm}

\begin{proof} This is just an exercise in changing notations. For $y\in\Z_p$, let $\kappa^y\colon\Gamma\rightarrow\C_\pr^*$
be the character $\gamma\mapsto\kappa(\gamma)^y$. Any such character can be extended, by linearity and continuity,
to a ring homomorphism $\kappa^y\colon\F_\pr[[\Gamma]]\rightarrow\C_\pr$, which is uniquely characterized by the following
property: if $\mu_\lambda$ denotes the measure on $\Gamma$ attached to $\lambda\in\F_\pr[[\Gamma]]$, then we have
\begin{equation}\label{eqintkapp}
\kappa^y(\lambda)=\int_\Gamma\kappa^y(\gamma)d\mu_\lambda=s(\lambda)(y) \end{equation}
(the last equality is the definition of $s$, as should be clear from the proof of Theorem \ref{teosmap}).

\noindent Let $\wt\alpha_i\colon \Z[[G_S]][[X]]\rightarrow\Lambda[[X]]$ be the homomorphism induced by composition of
$G_S\twoheadrightarrow G_\infty$ with the ring homomorphism $\wt\omega_\pr^i\colon W[[G_\infty]]\rightarrow\Lambda$
(as explained in Section \ref{sssStickcycl}). Definition \ref{DefSticFuncCyc} then becomes
$\Theta_\infty(X,\wt\omega_\pr^i)=\wt\alpha_i(\Theta_{\calf_S/F,S})(X)$. Moreover, letting $\alpha_i$ denote
the reduction of $\wt\alpha_i$ modulo $p$, \eqref{eqxis} yields the equality $\xi_{(y,i)}=\kappa^y\circ\alpha_i$.
For proving \eqref{eqsinnthet}, one just has to check
\[ L_\pr(X,-y,\omega_\pr^i)=\xi_{(y,-i)}(\Theta_{\calf_S/F,S})(X)=
\kappa^y\big(\alpha_{-i}(\Theta_{\calf_S/F,S})\big)(X)=s_X(\Theta_\infty(X,\wt\omega_\pr^{-i}))(y)\ .\]
The first equality is Theorem \ref{thmtetpr} and the last one is an easy consequence of \eqref{eqintkapp}.
\end{proof}

\section{Fitting ideals for Iwasawa modules}\label{SecFitt}
In this section we consider the algebraic aspect of the theory, i.e., Fitting ideals of Iwasawa modules associated
with the $\pr$-cyclotomic extension. Here the Stickelberger element will appear as a generator of Fitting ideals
of ($\chi$-parts of) class groups; the final link between the algebraic and the analytic side will be provided by
the Iwasawa Main Conjecture of Section \ref{SecIMC}.

Let $\ov{\F}$ be an algebraic closure of $\F$ and fix a topological generator $\g$ of the Galois group
$G_\F:=\Gal(\ov{\F}/\F)$ (the {\em arithmetic Frobenius}). For any field $L$ we denote by $L^{ar}$ the composition
$\ov{\F}L$ (i.e., the arithmetic extension of $L$): if $L/F$ is finite, then $\Gal(\ov{\F}L/L)\simeq G_\F$.
The arithmetic extension $F^{ar}$ is unramified at any prime and disjoint from $\calf$ (which is a
{\em geometric} extension), so $\Gal(\calf^{ar}/F)\simeq G_\infty\times G_\F$.

\subsection{Iwasawa modules in the $\pr$-cyclotomic extension}\label{NotIwaMod}
For any finite extension $L/F$, we let $\calCl(L)$ be the group of classes of degree zero divisors and we denote by $X_L$
the projective curve (defined over $\F$) associated with $L$. Let
\[ T_p(L):=T_p(Jac(X_L)(\ov{\F})) \]
be the $p$-Tate module of the $\ov{\F}$-points of the Jacobian of the curve $X_L$.
A first task is to compute the Fitting ideals of the modules
\[ \ov{C}_n:=\calCl(F_n^{ar})\{p\} \quad ,\quad C_n:=\calCl(F_n)\{p\} \quad{\rm and}\quad T_p(F_n)\]
as Iwasawa modules over some algebra containing $\Z_p[\G_n]$ (the $\{p\}$ indicates the $p$-part of the module; since
we shall mainly work with $p$-parts, the $\{p\}$ will often be omitted).
Recall that $T_p(F_n)\simeq \Hom(\Q_p/\Z_p,\ov{C}_n)$ as $\Z_p[\Delta][\G_n]$-modules.

Then we shall perform a limit on $n$ in order to provide a Fitting ideal in the Iwasawa algebra $\Lambda$.
This will be achieved by means of several maps induced by the natural norms and inclusions (see Section \ref{SecFittInf}).

For any prime $v$ of $F$, we let $F_n^{ar}(v)$ be the set of places of $F_n^{ar}$ lying above $v$ and we put
\[ H_{v,n}:= \Z_p[F_n^{ar}(v)] \ . \]
If $v$ is unramified in $F_n/F$, then $H_{v,n}$ is a $\Z_p$-free module of rank $|F_n^{ar}(v)|$ and (which is more relevant)
a $\Z_p[G_n]$-free module of rank $d_v:=\deg(v)$.

Since we shall work with Fitting ideals we recall one of the equivalent definition of these ideals (the one more
suitable for our computations).

\begin{defin}\label{DefFitt}
Let $M$ be a finitely generated module over a ring $R$. The {\em Fitting ideal} of $M$ over $R$, $\Fitt_R(M)$ is the ideal
of $R$ generated by the determinants of all the (minors of the) matrices of relations for a fixed set of generators of $M$.
\end{defin}

\subsubsection{Notation}
We remark that the integer $n\geq 1$ will always denote objects related with the $n$-th level $F_n$
of the $\pr$-cyclotomic extension. We shall work at a fixed finite level $n$ at first, then, in Section \ref{SecFittInf},
we let $n$ vary to compute limits.

\noindent To work with $\chi$-parts we extend our coefficients to $W$ by considering
$W\otimes_{\Z_p} M$ for any module $M$.
We will apply the decomposition \eqref{eqMchi} to $\liminv W\otimes_{\Z_p} C_n$,
$\liminv W\otimes_{\Z_p} \ov{C}_n$ or $\liminv W\otimes_{\Z_p} T_p(F_n)$.
To lighten notations we omit the $_{\Z_p}$; all tensor products will be defined
over $\Z_p$ unless we specify otherwise. For the same purpose whenever we have a map $\eta$ defined on a module $M$ we
shall still denote by $\eta$ the induced map ${\rm id}_W\otimes \eta$ on the module $W\otimes M$.

Finally, for any finite group $U$, we put
\[ n(U) := \sum_{h\in U} h \in \Z[U] \]
(this will mainly appear in the results on the $\chi_0$-part).

\subsection{Fitting ideals for the Tate module (I): finite level}\label{SecFittTate}
Let $\Fr_v$ denote the Frobenius of $v$ in $G_n$: it is the unique Frobenius attached to $v$ if $v$ is unramified, or
any lift of the Frobenius $\Fr_v\in G_n/I_{v,n}$ to $G_n$ if $v$ is ramified (this construction is easily seen to be
independent from the choice of the lift). In particular $\Fr_\pr=1$ (because $\pr$ is totally ramified in $F_n/F$) and
$\Fr_\infty=1$ as well because $\infty$ is totally split in $F_n^{I_{\infty,n}}/F$. We define the {\em Euler factor} at $v$ as
\[ e_v(X):=1-\Fr_v^{-1}X^{d_v} \ ,\]
where $X$ is a variable which will often be specialized to $\g^{-1}$, so we also put
\[ e_v:=e_v(\g^{-1})=1-\Fr_v^{-1}\g^{-d_v} \ .\]

\noindent The next result is exactly \cite[Lemmas 2.1 and 2.2]{GP2} for $F_n/F$.

\begin{lem}\label{FittHv}
Let $v$ be a place of $F$, then: \begin{enumerate}
\item if $v$ is unramified, we have $\Fitt_{\Z_p[G_n][[G_\F]]}(H_{v,n})=(e_v)$;
\item $\Fitt_{\Z_p[G_n][[G_\F]]}(H_{\infty,n})=(e_\infty, Aug_{\infty,n})=(1-\g^{-1}, Aug_{\infty,n})$;
\item $\Fitt_{\Z_p[G_n][[G_\F]]}(H_{\pr,n})=(e_\pr, Aug_{\pr,n})=(1-\g^{-d}, Aug_{\pr,n})$,
\end{enumerate}
where $Aug_{v,n}$ is the augmentation ideal associated to $I_{v,n}$, i.e., $Aug_{v,n}:=(\tau-1\,,\ \tau\in I_{v,n}\,)$.
\end{lem}

\begin{rem}\label{IsoHv}
If the prime $v$ is unramified, then the module $H_{v,n}$ is cyclic over the ring $\Z_p[G_n][[G_\F]]$ and one has
\[ H_{v,n} \simeq \Z_p[G_n][[G_\F]]/\Fitt_{\Z_p[G_n][[G_\F]]}(H_{v,n}) = \Z_p[G_n][[G_\F]]/(e_v)\ .\]
If we consider ramified primes then the same holds over the ring $\Z_p[G_n/I_{v,n}][[G_\F]]$. Hence
\[ H_{\infty,n} \simeq \Z_p[G_n][[G_\F]]/(e_\infty,Aug_{\infty,n}) = \Z_p[\Delta/\F^*\times \G_n][[G_\F]]/(e_\infty) \]
and
\[ H_{\pr,n} \simeq \Z_p[G_n][[G_\F]]/(e_\pr,Aug_{\pr,n}) = \Z_p[[G_\F]]/(e_\pr) \ .\]
\end{rem}

Let $\Sigma$ be a finite set of places of $F$ disjoint from $S$ and, for any $n$, put $\ov{S}_n:=F_n^{ar}(S)$
(resp. $\ov{\Sigma}_n:=F_n^{ar}(\Sigma)$) for the set of places of $F_n^{ar}$ lying above places in $S$ (resp. $\Sigma$).
Con\-si\-der the Deligne's Picard 1-motive $\calm_{\ov{S}_n,\ov{\Sigma}_n}$ associated to $F_n^{ar}$, $\ov{S}_n$ and
$\ov{\Sigma}_n$; it is represented by a group homomorphism
\[ Div^0(\ov{S}_n) \longrightarrow Jac_{\ov{\Sigma}_n}(X_{F_n})(\ov{\F}) \ ,\]
where $Div^0(\ov{S}_n)$ is the kernel of the degree map $\Z[\ov{S}_n]\rightarrow \Z$ and $Jac_{\ov{\Sigma}_n}(X_{F_n})$
is the extension of the Jacobian of $X_{F_n}$ by a torus (for more details on the definition of $\calm_{\ov{S}_n,\ov{\Sigma}_n}$
and its properties we refer the reader to \cite[Section 2]{GP1}).

We shall be working with the $p$-part of class groups, hence (by \cite[Remark 2.7]{GP1}) there is no contribution
from the toric part of $Jac_{\ov{\Sigma}_n}(X_{F_n})$. Therefore we can neglect the set $\Sigma$ (i.e., assume it is empty)
in what follows and focus simply on our $S=\{\pr,\infty\}$.
The multiplication by $p$ map
\[ Div^0(\ov{S}_n)\otimes_\Z \Z/p^m \longrightarrow Div^0(\ov{S}_n)\otimes_\Z \Z/p^{m-1} \]
induces a surjective map on the $p^m$-torsion of $\calm_{\ov{S}_n}:=\calm_{\ov{S}_n,\emptyset}$.
Thus one defines the {\em $p$-adic Tate module} of $\calm_{\ov{S}_n}$ as
\[ T_p(\calm_{\ov{S}_n}):= \il{m} \calm_{\ov{S}_n}[p^m] \]
(see \cite[Definitions 2.5 and 2.6]{GP1}). With this notations, in our setting, the main result of \cite{GP1} reads as

\begin{thm}{\rm (Greither-Popescu \cite[Theorem 4.3]{GP1})}\label{ThmFittStick}
One has
\begin{equation}\label{FittStick}
\Fitt_{\Z_p[G_n][[G_\F]]} (T_p(\calm_{\ov{S}_n})) = (\Theta_{F_n/F,S}(\g^{-1})):= (\Theta_n(\g^{-1}))
\end{equation}
where $\Theta_n(X)$ is the Stickelberger element
\begin{equation}\label{DefStick}
\Theta_n(X)= \prod_{v\notin S} (1-\Fr_v^{-1}X^{d_v})^{-1} \in \Z[G_n][[X]]
\end{equation}
(see \eqref{tetn}, with $d_v:=\deg(v)$).
\end{thm}

The relation between $\calm_{\ov{S}_n}$ and degree zero divisors with support in $\ov{S}_n$ leads to an exact sequence
\begin{equation}\label{SeqL}
0 \rightarrow T_p(F_n) \rightarrow T_p(\calm_{\ov{S}_n}) \rightarrow Div^0(\ov{S}_n)\otimes_\Z \Z_p \rightarrow 0 \ ,
\end{equation}
(see \cite[after Definition 2.6]{GP1}).

\noindent Our first task is to compute the Fitting ideal (over $\Z_p[G_n][[G_\F]]$)
of $T_p(F_n)$ and then project into $\Z_p[G_n]$ by specializing at $\g^{-1}=1$.
To do this we have to study the $\chi$-parts of the module $D_n:=Div^0(\ov{S}_n)\otimes_\Z \Z_p$ using the fact that
it is contained in $\Z_p[\ov{S}_n]=\oplus_{v\in S}H_{v,n}$. In most cases we will be able to compute
Fitting ideals using short exact sequences, while for the ``trivial'' component we will have to look for a resolution
\[ 0 \rightarrow D_n(\chi_0) \rightarrow X_3 \rightarrow X_4 \rightarrow 0 \]
that will fit in the sequence \eqref{SeqL} transforming it in a 4-term exact sequence to which we can apply \cite[Lemma 2.4]{GP2}.

\subsubsection{The $\chi$-parts of $D_n$}\label{ChiPartsL}
As seen in Remark \ref{IsoHv}, we have
\[ H_{\infty,n} \simeq \Z_p[G_n/I_{\infty,n}][[G_\F]]/(1-\g^{-1}) \simeq \Z_p[\Delta/\F^*\times \G_n] \ .\]
Therefore the $\chi$-parts depend on the values of $\chi$ on the elements of $\F^*=I_{\infty,n}$ and we have
\begin{equation}\label{ChiHinfty}
(W\otimes H_{\infty,n})(\chi) \simeq \left\{ \begin{array}{cl} 0 & {\rm if}\ \chi\ {\rm is \ odd}\\
\ & \\
W[\G_n] & {\rm if}\ \chi\ {\rm is \ even} \end{array} \right. \ . \end{equation}

\noindent Since there is no action of $\Delta$ on $H_{\pr,n} \simeq \Z_p[[G_\F]]/(e_\pr)$, we have
\begin{equation}\label{ChiHpr} (W\otimes H_{\pr,n})(\chi) \simeq \left\{ \begin{array}{cl} 0 & {\rm if}\ \chi\neq \chi_0\\
\ & \\
W[[G_\F]]/(e_\pr) & {\rm if}\ \chi=\chi_0 \end{array} \right. \ . \end{equation}

\noindent For any $\chi\neq \chi_0$ we can also observe that
\[ (W\otimes D_n)(\chi) = \Ker\{ e_\chi(W\otimes (H_{\pr,n} \oplus H_{\infty,n}))
\longrightarrow e_\chi (W\otimes \Z_p) = 0\} \simeq (W\otimes H_{\infty,n} )(\chi) \ , \]
so we are left with the trivial component $(W\otimes D_n)(\chi_0)$.

The degree map on $H_{\pr,n}$ provides a decomposition
\[ H_{\pr,n} \simeq (1-\g^{-1})H_{\pr,n} \oplus \Z_p \ ,\]
where $(1-\g^{-1})H_{\pr,n}$ is obviously in the kernel of the degree map on $D_n$ as well. For the trivial component we have
$e_{\chi_0}(W\otimes H_{\pr,n})= e_{\chi_0}(W\otimes (1-\g^{-1})H_{\pr,n}) \oplus W \cdot 1_{H_\pr}$ (where $1_{H_\pr}$ is the
unit element of $H_{\pr,n}$), and the map
\[ e_{\chi_0} (W\otimes (1-\g^{-1})H_{\pr,n} )\oplus e_{\chi_0}(W\otimes H_{\infty,n} ) \longrightarrow
e_{\chi_0} (W\otimes D_n) \]
given by
\[ (\alpha, \beta) \to (\alpha - \deg(\beta) 1_{H_\pr},\beta) \in e_{\chi_0}(W\otimes (H_{\pr,n} \oplus H_{\infty,n})) \]
is easily seen to be an isomorphism of $W[\G_n][[G_\F]]$-modules. Hence we obtain
\begin{equation}\label{ChiLnd}
(W\otimes D_n)(\chi) \simeq \left\{ \begin{array}{cl} 0 & {\rm if}\ \chi\ {\rm is\ odd} \\
\ & \\
W[\G_n] & {\rm if}\ \chi\neq \chi_0 \ {\rm is \ even} \\
\ & \\
(1-\g^{-1})W[[G_\F]]/(1-\g^{-d}) \oplus W[\G_n] & {\rm if}\
\chi = \chi_0 \end{array} \right. \ . \end{equation}

\subsubsection{Computation of Fitting ideals (I)}
The previous descriptions of the $(W\otimes D_n)(\chi)$ allow the
first computation of Fitting ideals for the Tate modules.

\begin{prop}\label{FittTFn} We have
\[ \Fitt_{W[\G_n][[G_\F]]} ((W\otimes T_p(F_n))(\chi)) =
\left\{ \begin{array}{cl} (\Theta_n(\g^{-1},\chi)) & {\rm if}\ \chi\ {\rm is\ odd} \\
\ & \\
\left(\frac{\Theta_n(\g^{-1},\chi)}{1-\g^{-1}} \right) & {\rm if}\
\chi\neq \chi_0 \ {\rm is \ even} \end{array} \right. \]
and
\[ \Fitt_{W[\G_n][[G_\F]]} ((W\otimes T_p(F_n))(\chi_0)^*) =
\frac{\Theta_n(\g^{-1},\chi_0)}{1-\g^{-1}} \left( 1,\frac{n(\G_n)}{\nu_d}\right) \]
(where $\,^*$ denotes the $\Z_p$-dual and $\nu_d:=\frac{1-\g^{-d}}{1-\g^{-1}}$).
\end{prop}

\begin{proof} We split the proof in three parts, depending on the type of the character $\chi\in \widehat{\Delta}$.\\
{\bf Case 1: $\chi$ is odd.} Since $e_\chi D_n=0$, we have an isomorphism
\[ (W\otimes T_p(F_n))(\chi) \simeq (W\otimes T_p(M_{\ov{S}_n}))(\chi)\ .\]
Hence, by Theorem \ref{ThmFittStick} above,
\[ \Fitt_{W[\G_n][[G_\F]]} ((W\otimes T_p(F_n))(\chi)) = (\Theta_n(\g^{-1},\chi)) \]
(because, by equation \eqref{EqChiTheta}, $e_{\chi}\Theta_n(X)=\Theta_n(X,\chi)e_{\chi}$).

\noindent {\bf Case 2: $\chi\neq\chi_0$ is even.} In this case
$(W\otimes D_n)(\chi)\simeq W[\G_n][[G_\F]]/(1-\g^{-1})$
is a cyclic $W[\G_n][[G_\F]]$-module. We have an exact
sequence
\[ (W\otimes T_p(F_n))(\chi) \iri (W\otimes T_p(M_{\ov{S}_n}))(\chi) \sri  W[\G_n][[G_\F]]/(1-\g^{-1}) \]
to which we can apply \cite[Lemma 3]{CG} to get
\[ \Fitt_{W[\G_n][[G_\F]]} ((W\otimes T_p(F_n))(\chi)) (1-\g^{-1}) = (\Theta_n(\g^{-1},\chi)) \ .\]

\noindent{\bf Case 3: $\chi=\chi_0$.}
Consider the resolution for $D_n(\chi_0)$ provided by the sequences
\begin{equation}\label{ResHinfty}
\Z_p[\G_n] \iri \Z_p[\G_n][[G_\F]]/(e_\infty) \sri \Z_p[\G_n][[G_\F]]/(e_\infty,n(I_{\infty,n}))
\end{equation}
and
\begin{equation}\label{ResHpr}
(1-\g^{-1}) \Z_p[[G_\F]]/(e_\pr) \!\! \iri \!\! (1-\g^{-1}) \Z_p[\G_n][[G_\F]]/(e_\pr)
\!\! \sri \!\! (1-\g^{-1}) \Z_p[\G_n][[G_\F]]/(e_\pr,n(I_{\pr,n}))
\end{equation}
where the map on the left is given by $1_{H_v}\to n(I_{v,n})$ ($v=\infty$, $\pr$). To check exactness one simply observes
that all the modules involved are $\Z_p$-free modules and counts ranks. Joining the sequences \eqref{ResHinfty} and \eqref{ResHpr}
with the sequence \eqref{SeqL} and tensoring with $W$ (limiting ourselves to the $\chi_0$-part), we find
\[ (W\otimes T_p(F_n))(\chi_0) \iri  (W\otimes T_p(M_{\ov{S}_n}))(\chi_0) \rightarrow
W[\G_n][[G_\F]]/(e_\infty)\oplus W[\G_n][[G_\F]]/(\nu_d) \]
\vspace{-1truecm}
\[ \hspace{7.5truecm} \xymatrix{ \ar@{->>}[d] \\ \ } \]
\vspace{-.4truecm}
\[ \hspace{5truecm} W[\G_n][[G_\F]]/(e_\infty,n(I_{\infty,n}))\oplus W[\G_n][[G_\F]]/(\nu_d ,n(I_{\pr,n})) \ . \]
We note that the assumptions of \cite[Lemma 2.4]{GP2} hold for the previous sequence (actually they hold before tensoring
with $W$ but the computation of Fitting ideals is not affected by that, moreover we use the full ring $R=\Z_p[\G_n][[G_\F]]$
instead of the $R'$ of the original paper but the lemma still holds as the authors mention right before stating it).
Indeed $T_p(M_{\ov{S}_n})(\chi_0)$ is finitely generated and free over $\Z_p$ (so it has no nontrivial finite submodules) and it is
$\G_n$-cohomologically trivial by the proof of \cite[Theorem 3.9]{GP1}: hence it is of projective dimension 1 over
$\Z_p[\G_n][[G_\F]]$ by \cite[Proposition 2.2 and Lemma 2.3]{Po}. The other 3 modules are finitely generated and free
over $\Z_p$ and, obviously, $\Z_p[\G_n][[G_\F]]/(e_\infty)\oplus \Z_p[\G_n][[G_\F]]/(\nu_d)$ has
projective dimension at most 1. Therefore we can apply \cite[Lemma 2.4]{GP2} which immediately yields the final
statement of the proposition.
\end{proof}

\subsection{Fitting ideals for class groups}
There are deep relations between $T_p(F_n)$ and the modules
$C_n:=\calCl(F_n)\{p\}$ (the ones we are primarily interested in). Indeed, as noted at the beginning
of \cite[Section 3]{GP2}, the $\Z_p$-dual $T(F_n)^*$ of $T(F_n)$ verifies
\begin{equation}\label{GPFromTpToCn} (T_p(F_n)^*)_{G_\F} \simeq \calCl(F_n)\{p\}^\vee =C_n^\vee\ ,
\end{equation}
i.e., its $G_\F$-coinvariants are isomorphic to the Pontrjagin dual of $C_n$.
Another one is provided by the following

\begin{lem}\label{FromTpToCn}
We have an isomorphism of $\Z_p[G_n]$-modules
\[ C_n \simeq T_p(F_n)/(1-\g^{-1})T_p(F_n)=T_p(F_n)_{G_\F} \ .\]
\end{lem}

\begin{proof} We recall that $\ov{C}_n$ is the
$p$-Sylow of $\calCl(F_n^{ar})$, hence it is divisible and isomorphic to $(\Q_p/\Z_p)^r$ for some
$r\leq g_n$ (where $g_n$ is the genus of $X_{F_n}$).
Obviously $\ov{C}_n$ is a $\Z_p[G_n][[G_\F]]$-module and we have
\[ T_p(F_n)=\plim{m}\,\ov{C}_n\,[p^m] \simeq \Hom(\Q_p/\Z_p,\ov{C}_n) \]
(the Galois action on the module on the right is the usual one $(\sigma\cdot f)(y):=\sigma f(\sigma^{-1}y)$
for any $f\in \Hom(\Q_p/\Z_p,\ov{C}_n)$).

\noindent By Lang's theorem (see, for example, \cite[Chapter VI, \S 4]{Se}) we have an exact sequence
\[ 0\rightarrow C_n\rightarrow\ov{C}_n\,
{\buildrel {1-\g^{-1}}\over{-\!\!\!-\!\!\!-\!\!\!-\!\!\!-\!\!\!\!\longrightarrow}}\,\ov{C}_n\rightarrow 0 \ .\]
Applying the functor $\Hom(\Q_p/\Z_p,*)$ (a similar argument can be found in \cite[Lemma 4.1]{An}) one gets
\[ 0\rightarrow T_p(F_n)\,{\buildrel {1-\g^{-1}}\over{-\!\!\!-\!\!\!-\!\!\!-\!\!\!-\!\!\!\!\longrightarrow}}
\,T_p(F_n)\rightarrow \Ext^1(\Q_p/\Z_p,C_n)\rightarrow 0 \]
because $\Hom(\Q_p/\Z_p,C_n)=0$ ($C_n$ is finite) and $\Ext^1(\Q_p/\Z_p,\ov{C}_n)=0$ ($\ov{C}_n$ is divisible).
Now from the usual short exact sequence
\[ 0\rightarrow\Z_p\rightarrow\Q_p\rightarrow\Q_p/\Z_p\rightarrow 0 \ ,\]
applying $\Hom(*,C_n)$, we obtain
\[ \Hom(\Q_p,C_n)=0 \rightarrow \Hom(\Z_p,C_n)\simeq C_n\rightarrow \Ext^1(\Q_p/\Z_p,C_n)\rightarrow \Ext^1(\Q_p,C_n)=0 \ .\]
Therefore we have an isomorphism
\[ C_n\simeq \Ext^1(\Q_p/\Z_p,C_n) \simeq T_p(F_n)/(1-\g^{-1})T_p(F_n)=T_p(F_n)_{G_\F} \ .\qedhere\]
\end{proof}

\begin{rem} \label{Duals}
Equation \eqref{GPFromTpToCn} and Lemma \ref{FromTpToCn} together yield
\[ (T_p(F_n)^*)_{G_\F}\simeq C_n^\vee \simeq (T_p(F_n)_{G_\F})^\vee \ .\]
A general statement of this type appears in \cite[Lemma 5.18]{GP1}.\end{rem}

\begin{defin} Let $\chi\in \Hom(\Delta,W^*)$ with $\chi\neq\chi_0$ and $n\in\N\cup\{\infty\}$, we define the
{\em modified Stickelberger function} as
\[ \Theta_n^{\#}(X,\chi):=\left\{\begin{array}{ll}
\Theta_n(X,\chi) & {\rm if\ }\chi\ {\rm is\ odd}\\
\ & \ \\
\ds{\frac{\Theta_n(X,\chi)}{1-X}} & {\rm if\ }\chi\ {\rm is\ even}
\end{array}\right. \ . \]
\end{defin}

\noindent Consider the projection map
$\pi_{G_\F}:W[\G_n][[G_\F]]\rightarrow W[\G_n]$ which maps $\g$ to $1$. The properties of Fitting ideals and
Lemma \ref{FromTpToCn} yield
\begin{cor}\label{cor3.4} For $\chi\neq \chi_0$  we have
\[ \Fitt_{W[\G_n]}((W\otimes C_n)(\chi))=(\Theta_n^{\#}(1,\chi)) \ .\]
\end{cor}

\noindent While the isomorphism \eqref{GPFromTpToCn} leads to
\begin{cor}\label{FittTrivChar} For the trivial character $\chi_0$ we have
\[ \Fitt_{W[\G_n]}((W\otimes C_n^\vee )(\chi_0)\otimes_W Q(W))=
\frac{\Theta_n(X,\chi_0)}{1-X}|_{X=1}\left(1,\frac{n(\G_n)}{d}\right) \]
(where $Q(W)$ is the quotient field of $W$).
\end{cor}

\subsection{Fitting ideals for Tate modules (II): infinite level}\label{SecFittInf}
Consider the Iwasawa tower $\mathcal{F}/F$ and let $\varphi^{n+1}_n:X_{F_{n+1}}\rightarrow X_{F_n}$ be the morphism of
curves corresponding to the field extension $F_{n+1}/F_n$; it is a $\G^{n+1}_n:=\Gal(F_{n+1}/F_n)$ Galois cover
totally ramified at $\pr$. As before $\chi$ denotes an element of $\Hom(\Delta,W^*)$.

\noindent We have morphism $i^n_{n+1}:T_p(F_n)\hookrightarrow T_p(F_{n+1})$ (induced by the natural map from
$\ov{C}_n$ to $\ov{C}_{n+1}$) and a map $N^{n+1}_n:T_p(F_{n+1})\rightarrow T_p(F_n)$ induced by the norm map
from $\ov{C}_{n+1}$ to $\ov{C}_n$. Observe that, for any $n$, $N^{n+1}_n\circ i^n_{n+1}=q^d$.

\noindent In this section we shall meet various other maps induced by norms (resp. inclusions) on different modules/objects:
by abuse of notations we shall denote all of them by $N^{n+1}_n$ (resp. $i^n_{n+1}$), when we need some distinction
between them we shall write $N(\bullet)^{n+1}_n$ (resp. $i(\bullet)^n_{n+1}$) to denote the map defined on the
objects $\bullet$ or $T_p(\bullet)$.

\subsubsection{Norm and inclusion maps}
We have an inclusion $i^n_{n+1}:T_p(\calm_{\ov{S}_n})\iri T_p(\calm_{\ov{S}_{n+1}})$
such that $T_p(\calm_{\ov{S}_{n+1}})^{\G^{n+1}_n}=i^n_{n+1}(T_p(\calm_{\ov{S}_n}))$ by \cite[Theorem 3.1]{GP1}.
We also have a natural norm map
$N(\calm)^{n+1}_n:T_p(\calm_{\ov{S}_{n+1}})\rightarrow T_p(\calm_{\ov{S}_n})$.

\begin{lem}\label{LemN(calm)} The norm map $N(\calm)^{n+1}_n$ is surjective and its kernel is
$I_{\G^{n+1}_n} T_p(\calm_{\ov{S}_{n+1}})$ where $I_{\G^{n+1}_n}$ is the augmentation ideal (i.e., generated by
$\{\sigma-1\,:\,\sigma\in\G^{n+1}_n\}$).
\end{lem}

\begin{proof} By \cite[Theorem 3.9]{GP1} (in particular, its proof) we have that $T_p(\calm_{\ov{S}_{n+1}})$ is
$\G^{n+1}_n$-co\-ho\-mo\-lo\-gi\-cal\-ly trivial, i.e.,
\[ \wh{H}^i(\G^{n+1}_n,T_p(\calm_{\ov{S}_{n+1}}))=0 \quad\forall\,i \ .\]
For $i=0$ we have that
\[ T_p(\calm_{\ov{S}_{n+1}})^{\G^{n+1}_n}=N(\calm)^{n+1}_n(T_p(\calm_{\ov{S}_{n+1}})) \ ,\]
but, as recalled above, $T_p(\calm_{\ov{S}_{n+1}})^{\G^{n+1}_n}=T_p(\calm_{\ov{S}_n})$,
therefore the norm map is surjective.

\noindent With $i=-1$ we obtain that the kernel of $N(\calm)^{n+1}_n$ is given by the augmentation module
$I_{\G^{n+1}_n} T_p(\calm_{\ov{S}_{n+1}})$.
\end{proof}

\noindent We have a commutative diagram of short exact sequences
\begin{equation}\label{DiagNorm1}
\xymatrix {0 \ar[r] & T_p(F_{n+1}) \ar[r] \ar[d]^{N(F)^{n+1}_n} & T_p(\calm_{\ov{S}_{n+1}})\ar[r] \ar[d]^{N(\calm)^{n+1}_n} &
D_{n+1} \ar[r] \ar[d]^{N(D)^{n+1}_n} \ar[r] & 0 \\
0 \ar[r] & T_p(F_n) \ar[r] & T_p(\calm_{\ov{S}_n}) \ar[r] & D_n \ar[r] & 0 }
\end{equation}
where all vertical maps are induced by norms: in particular note that $N(D)^{n+1}_n$ corresponds to the natural
map on divisors
\[ \Z_p[F_{n+1}^{ar}(\infty)]\oplus\Z_p[F_{n+1}^{ar}(\pr)]\rightarrow
\Z_p[F_{n}^{ar}(\infty)]\oplus\Z_p[F_{n}^{ar}(\pr)] \ .\]

\begin{lem}\label{LemN(L)} Let $\chi\neq\chi_0$ and $n\geq 1$. Then
\[ \Ker(N(D)^{n+1}_n)(\chi)=\left\{\begin{array}{cl}
0 &{\rm \ if}\ \chi{\rm \ is\ odd}\\
\ & \\
I_{\G^{n+1}_n} (W\otimes D_{n+1})(\chi) &{\rm \ if}\ \chi{\rm\ is\ even} \\
\end{array}\right. \ . \]
Moreover, the map $N(D)^{n+1}_n$ is surjective.
\end{lem}

\begin{proof} The last assertion immediately follows from the surjectivity of $N(\calm)^{n+1}_n$ and the snake lemma
sequence of diagram \eqref{DiagNorm1}. Now consider the same diagram but with $\chi$-parts and tensored with $W$ and
note that the maps $N(\calm)^{n+1}_n$ and $N(D)^{n+1}_n$ remain surjective.\\
From the computations in Section \ref{ChiPartsL} the case $\chi$ odd is obvious. When $\chi\neq\chi_0$ is even we have
$e_{\chi}(W\otimes D_{n+1})\simeq W[\G_{n+1}]$ and the norm corresponds to the projection
$W[\G_{n+1}]\rightarrow W[\G_n]$ which has kernel $I_{\G^{n+1}_n}$.
\end{proof}

The previous two lemmas lead to similar statements for the map $N(F)^{n+1}_n(\chi)$.

\begin{prop}\label{CohoTriv} Let $\chi\neq\chi_0$, then $(W\otimes T_p(F_n))(\chi)$ is $\G_n$ and
$\G^{n+1}_n$-cohomologically trivial and a free $W[\G_n]$-module. In particular
\[ (W\otimes T_p(F_{n+1}))(\chi)^{\G^{n+1}_n}=(W\otimes T_p(F_n))(\chi)\ ,\]
$N(F)^{n+1}_n(\chi)$ is surjective and
$\Ker(N(F)^{n+1}_n)(\chi)=I_{\G^{n+1}_n}(W\otimes T_p(F_{n+1}))(\chi)$.
Moreover it is also a $W[\G_n][[G_\F]]$-module of projective dimension less than or equal to one.
\end{prop}

\begin{proof} Consider the short exact sequence
\[ 0\rightarrow W\otimes T_p(F_n)\rightarrow W\otimes T_p(\calm_{\ov{S}_n}) \rightarrow W\otimes D_n\rightarrow 0 \ .\]
Since $D_n(\chi)$ is $0$ or $W[\G_n]$, it is $W[\G_n]$-free and cohomologically trivial, while $W\otimes T_p(\calm_{\ov{S}_n})$ is
also $W[\G_n]$-free and cohomologically trivial by \cite[Theorem 3.9]{GP1}. Thus $(W\otimes T_p(F_n))(\chi)$ is
projective over $W[\G_n]$ and cohomologically trivial. Now, since $\G_n$ is a $p$-group, $W[\G_n]$ is a local ring and
projective modules coincide with free modules. The cohomological triviality over $\G^{n+1}_n$
is similar and straightforward.

\noindent The assertion on $\Ker(N(F)^{n+1}_n)(\chi)$ comes from the triviality of the $\wh{H}^1$.
Now take $\G^{n+1}_n$-invariants in the sequence for even characters (for odd ones there is nothing to prove)
\[(W\otimes T_p(F_{n+1}))(\chi) \iri (W\otimes T_p(\calm_{\ov{S}_{n+1}}))(\chi) \sri (W\otimes D_{n+1})(\chi) \simeq W[\G_{n+1}] \ , \]
to get
\[(W\otimes T_p(F_{n+1}))(\chi)^{\G^{n+1}_n} \iri (W\otimes T_p(\calm_{\ov{S}_n}))(\chi) \sri (W\otimes D_n)(\chi) \simeq W[\G_n] \]
(using Lemmas \ref{LemN(calm)} and \ref{LemN(L)}). Thus $(W\otimes T_p(F_{n+1}))(\chi)^{\G^{n+1}_n}=(W\otimes T_p(F_n))(\chi)$
and, since the $\wh{H}^0$ is trivial, $N(F)^{n+1}_n(\chi)$ is surjective.

\noindent The last statement of the lemma follows from \cite[Proposition 2.2 and Lemma 2.3]{Po} (see also
\cite[Proposition 5.3]{DoPo}) because $(W\otimes T_p(F_n))(\chi)$, being free, has no nontrivial finite $W[\G_n]$-submodule.
\end{proof}

\begin{rem}\label{RemChi0} For $\chi=\chi_0$ the modules $(W\otimes T(\calm_{\ov{S}_n}))(\chi_0)$ and $W[\G_n]$ are
still $\G_n$-co\-ho\-mo\-lo\-gi\-cal\-ly trivial and we have the short exact sequence
\[ (W\otimes T_p(F_n))(\chi_0)\iri (W\otimes T(\calm_{\ov{S}_n}))(\chi_0)\sri \frac{(1-\g^{-1})W[[G_\F]]}{(1-\g^{-d})}\oplus W[\G_n] \ .\]
Since $F_n/F$ is disjoint from $\ov{\F}/\F$, the norm acts on $G_\F$ as multiplication by $[F_n:F]$. For any
subextension $E/K$ and for any $k\in \Z$, we have
\[ \wh{H}^k\left(\Gal(E/K),\frac{(1-\g^{-1})W[[G_\F]]}{(1-\g^{-d})}\right)\simeq
\wh{H}^{k+1}(\Gal(E/K),(W\otimes T_p(F_n))(\chi_0)) \]
and, in particular,
\[ \wh{H}^0(\Gal(E/K),(W\otimes T_p(F_n))(\chi_0)) \! \simeq \!
\frac{(1-\g^{-1})W[[G_\F]]}{(1-\g^{-d})}/|\Gal(E/K)|\frac{(1-\g^{-1})W[[G_\F]]}{(1-\g^{-d})} \,.\]
Therefore $(W\otimes T_p(F_n))(\chi_0)$ is not $\G_n$-cohomologically trivial (and not not necessarily $W[\G_n]$-free).
\end{rem}

\subsubsection{Computation of Fitting ideals (II)}
We define
\[ T_p(\calf)(\chi) :=\plim{n} (W\otimes T_p(F_n))(\chi) \]
(the limit is on the norm maps studied above). We recall that $T_p(F_n)=T_p(Jac(X_{F_n})(\ov{\F}))$,
so $T_p(\calf)(\chi)$ is a $W[[\G]][[G_\F]]$-module.

\begin{prop}\label{FinGen1} For $\chi\neq\chi_0$,
$T_p(\calf)(\chi)$ is a finitely generated torsion $W[[\G]][[G_\F]]$-module.
\end{prop}

\begin{proof} The ideals $\gotI_n$ (defined in Section \ref{SecIwaAlg}) form an open filtration for $W[[\G]]$ and we note that
$\gotI_n=\plim{m}I_{\G^{m+n}_n}$ (where $I_{\G^{m+n}_n}$ denotes the augmentation ideal associated to
$\G^{m+n}_n=\Gal(F_{n+m}/F_n)$).
From Proposition \ref{CohoTriv} we have that
\[ T_p(\calf)(\chi)/(W[[\G]][[G_\F]]\otimes_{W[[\G]]}\gotI_n)T_p(\calf)(\chi) \simeq (W\otimes T_p(F_n))(\chi) \]
for all $n\geq 1$. The module on the right is finitely generated over
\[ W[\G_n][[G_\F]]= W[[\G]][[G_\F]]/(W[[\G]][[G_\F]]\otimes_{W[[\G]]}\gotI_n) \ ,\]
so, by a generalized Nakayama Lemma (see \cite[Corollary p. 226]{BaHo}), we have that $T_p(\calf)(\chi)$ is finitely generated
as a $W[[\G]][[G_\F]]$-module. Moreover $\Theta_n^\#(\g^{-1},\chi)(W\otimes T_p(F_n))(\chi)=0$
for any $n$ (by Proposition \ref{FittTFn}), hence
\[ \Theta_\infty^\#(\g^{-1},\chi)T_p(\calf)(\chi)=0 \ ,\]
i.e., the module $T_p(\calf)(\chi)$ is torsion.
\end{proof}

\noindent Therefore the Fitting ideal of the $W[[\G]][[G_\F]]$-module $T_p(\calf)(\chi)$ is well defined
and we have the following formula for it (for a similar result see \cite[Theorem 2.1]{GrKu},
but note the particular case of \cite[Remark 2.2\,(2)]{GrKu} which fits our setting).

\begin{thm}\label{FittCalmCalf} For $\chi\neq\chi_0$ we have
\[ \Fitt_{W[[\G]][[G_\F]]}(T_p(\calf)(\chi))=(\Theta_{\infty}^{\#}(\g^{-1},\chi)) \ .\]
\end{thm}

\begin{proof}
By the previous proposition we can find an $r\in\N$ such that the
following diagram commutes
\begin{equation}\label{DiagFitt1}
\xymatrix { 0 \ar[r] & B_{n+1} \ar[r] \ar[d]^{b^{n+1}_n} & W[\G_{n+1}][[G_\F]]^r \ar[r] \ar[d]^{\pi^{n+1}_n} &
(W\otimes T_p(F_{n+1}))(\chi)\ar[r]\ar[d]^{N(F)^{n+1}_n(\chi)}& 0 \\
0 \ar[r] & B_n \ar[r] & W[\G_n][[G_\F]]^r \ar[r] & (W\otimes T_p(F_n))(\chi)\ar[r]& 0 }
\end{equation}
(where the central map is the canonical projection).
The maps $\pi^{n+1}_n$ and $N(F)^{n+1}_n(\chi)$ are surjective and that their kernels are
$\left( I_{\G^{n+1}_n}W[\G_{n+1}][[G_{\ov{\F}}]]\right)^r$ and $I_{\G^{n+1}_n}(W\otimes T_p(F_{n+1}))(\chi)$.
The map between these two kernels is obviously surjective, hence, by the snake lemma sequence, we have
that $b^{n+1}_n$ is surjective as well.

\noindent Taking the inverse limit in diagram \eqref{DiagFitt1} (which verifies the Mittag-Leffler condition),
we obtain the exact sequence (recall $\L=W[[\G]]$)
\begin{equation}\label{EqFittCalmCalf} 0 \rightarrow B_\infty:=\plim{n} B_n \rightarrow \L[[G_\F]]^r
\rightarrow \plim{n} (W\otimes T_p(F_n))(\chi)=T_p(\calf)(\chi)
\rightarrow 0 \ .\end{equation} Recall that we can use any
$\beta_1,\ldots,\beta_r\in B_n$ as rows of a matrix
$M_{\beta_1,\ldots,\beta_r}\in Mat_r(W[\G_n])$ and
$\Fitt_{W[\G_n]}((W\otimes T_p(F_n))(\chi))$ is generated by the
$\det(M_{\beta_1,\ldots,\beta_r})$. The surjectivity of the maps
$b^{n+1}_n$ implies the same property for the induced maps
$b_n:B_\infty \sri B_n$, i.e., the ``relations'' at the infinite
level are all induced by ``relations'' already existing at lower
levels (the technical arguments of the final parts of \cite[Theorem 2.1]{GrKu} are not necessary here because of
the presence of just one ramified prime and our previous computations on kernels of norm maps).
Using the characterization of the Fitting ideal in Definition \ref{DefFitt},
it is easy to see that the surjectivity of the $b_n$ and the sequence
\eqref{EqFittCalmCalf} yield the desired result, i.e.,
\[ \begin{array}{cl}\Fitt_{\L[[G_\F]]}(T_p(\calf)(\chi))& =(\det(B_\infty))=\plim{n}(\det(B_n)) \\
\ &= \plim{n} \Fitt_{W[\G_n][[G_\F]]}((W\otimes T_p(F_n))(\chi)) \\
\ &= \plim{n} (\Theta_n^\#(\g^{-1},\chi)) = (\Theta_\infty^\#(\g^{-1},\chi))\ . \qquad\qquad\qquad\qquad\qedhere \end{array} \]
\end{proof}

\section{Iwasawa main conjecture for the $\pr$-cyclotomic extension}\label{SecIMC}
Consider now $W\otimes C_n=W\otimes Jac(X_{F_n})(\F)=W\otimes Pic^0(X_{F_n})(\F)$ as a $W[G_n]$-module and
the natural maps $i(C)_{n+1}^n:W\otimes C_n\rightarrow W\otimes C_{n+1}$ and
$N(C)^{n+1}_n:W\otimes C_{n+1}\rightarrow W\otimes C_n$. Denote by $\mathcal{C}$ the $W[[G_{\infty}]]$-module
$\plim{n}W\otimes C_n$ (defined, as usual, with respect the norm maps $N(C)^{n+1}_n$).

\noindent The main results of this section are the following

\begin{thm}\label{Cfinitgen}  Let $\chi\neq \chi_0$, then the module $\mathcal{C}(\chi):=\e_{\chi}\mathcal{C}$
is a finitely generated torsion $\L$-module.
\end{thm}

\noindent Therefore the Fitting ideal $\Fitt_{\L}(\mathcal{C}(\chi))$ is well defined and we have

\begin{thm}[Iwasawa Main Conjecture] \label{IMClevel}
Let $\chi\neq\chi_0$, then
\[ \Fitt_{\L}(\mathcal{C}(\chi))=(\Theta_{\infty}^{\#}(1,\chi)) \ .\]
\end{thm}

\begin{rem}
The theorem above allows us to compute the Fitting ideal of $\calc(\chi)$ as the inverse limit of the Fitting ideals
appearing in the (natural) filtration of $\calf$ given by the fields $F_n$. A different approach to the same problem
is provided in \cite[Section 5]{BBL1} where the authors use a filtration of $\Z_p^d$-extensions (a more general approach
and the fact that the limit is independent from the filtration are shown in \cite{BBL2}).
In that paper the statement of the Main Conjecture involves characteristic ideals but (for Iwasawa modules) they coincide
with Fitting ideals whenever the Fitting is principal (see, for example, \cite[Lemma 5.10]{BL}).
\end{rem}

\noindent Before going into the proofs of the above theorems, we need a crucial lemma.

\begin{lem}\label{BigLemma} Let $F_0\subset K\subset E\subset \mathcal{F}$, where $E/F$ is a finite extension
and the group $G:=\Gal(E/K)$ is a $p$-group. For any field $L\subset \calf$ we let $\pr_L$ be the unique prime of $L$ lying
above $\pr$, we recall that $\calCl(L)$ denotes the group of classes of degree zero divisors. We have the following properties:
\begin{enumerate}
\item the map $i^K_E:\CaCl^0(K)\rightarrow \CaCl^0(E)$ is injective;
\item there is an equality $\CaCl^0(E)^G=i^K_E(\CaCl^0(K))+\langle r\frac{|G|}{p^t}\pr_{E}-\frac{d}{p^t} i^K_E(v)\rangle$,
where $p^t:=(|G|,d)$, $v$ is a place of $K$ lying above a prime of $A$ of degree $r$ (prime with $p$)
and which is totally split in $E$. The second term disappears when we consider $\chi$-parts for nontrivial characters,
in particular, for $\chi\neq\chi_0$, we have
\[ ((W\otimes \CaCl^0(E))(\chi))^G=i^K_E((W\otimes \CaCl^0(K))(\chi))\ ;\]
\item the norm map $N^E_K:\CaCl^0(E)\rightarrow \CaCl^0(K)$ is surjective;
\item for $\chi\neq\chi_0$, we have $\Ker(N^E_K(\chi))=I_G (W\otimes \CaCl^0(E))(\chi)$.
\end{enumerate}
\end{lem}

\begin{proof} For simplicity we write $N$ and $i$ for $N^E_K$ and $i^K_E$ respectively.

\noindent (1) For any field $L$ write $P_L$ for the principal divisors of $L$. Consider the exact sequences
\begin{equation}\label{Eq1}
0\rightarrow \F^* \rightarrow E^* \rightarrow P_E \rightarrow 0
\end{equation}
and
\begin{equation}\label{Eq2}
0\rightarrow P_E \rightarrow Div^0(E) \rightarrow \CaCl^0(E) \rightarrow 0\ .
\end{equation}
Taking $G$-cohomology in \eqref{Eq1} one finds
\[ 0\rightarrow \F^* \rightarrow K^* \rightarrow P_E^G \rightarrow 0\rightarrow 0\rightarrow H^1(G,P_E)\rightarrow 0 \]
(because of Hilbert 90 and because $G$ is a $p$-group so $H^i(G,\F^*)=0$ for any $i\geqslant 1$), so, in particular,
$P_E^G=P_K$ and $H^1(G,P_E)=0$. The $G$-invariants of the sequence \eqref{Eq2} then fit into the diagram
\begin{equation}\label{DiagBigLemma}
\xymatrix { P_K \ar@{^(->}[r] \ar@{=}[d] & Div^0(K) \ar@{->>}[r] \ar[d]^i & \CaCl^0(K) \ar[d]^{i^K_E} \\
P_E^G =P_K \ar@{^(->}[r] & Div^0(E)^G \ar[r] & \CaCl^0(E)^G \ar[r] & H^1(G,P_E)=0 \ .}
\end{equation}
The injectivity of the central vertical map and the snake lemma sequence yield the desired injectivity of $i^K_E$.

\noindent (2) We need several steps:
\begin{enumerate}
\item[(a)] {\em The group $Div(E)^G/i(Div(K))$ is cyclic of order $[E:K]$ and
is generated by the class of $\pr_E$.}

\noindent Write $Div(K)=\oplus_v\Z v$ ($v$ runs through all the primes of $K$) and $Div(E)=\oplus_v H_v$ with
$H_v=\oplus_{w|v} \Z w$. If $v\neq\pr_K$ is unramified, then $H_v=\Z[G/G_v]w$, where $G_v$ is the decomposition subgroup
of $v$ in $G$, and obviously
\[ H_v^G=\Z i(v)\quad {\rm with}\quad i(v)=\sum_{\sigma\in G/G_v}\sigma w \ .\]
For the ramified place we have $\sigma(\pr_E)=\pr_E$ and $|G|\pr_E=i(\pr_K)$: this yields the statement.

\item[(b)] {\em The group $Div^0(E)^G/i(Div^0(K))$ is killed by $p^t:=(|G|,d)$.}

\noindent Take $D\in Div^0(E)^G$, by part (a) we can write $D=n\pr_E+D'$ with $D'\in i(Div(K))$. Since $D$ has degree
zero we have $-nd=\deg_E (D')=|G|\deg_K (D')$ and $\frac{|G|}{p^t}$ divides $n=\frac{|G|}{p^t}n'$. Therefore
\[ \begin{array}{cl} p^tD & = p^tn\pr_E +p^tD' = |G|n'\pr_E +p^tD' \\
\ & = n'\pr_K +p^tD' \in i(Div(K)) \ .\end{array} \]
Moreover
\[ \begin{array}{cl} \deg_K(n'\pr_K +p^tD')& =n'd +p^t\deg_K(D')=\ds{\frac{p^tn}{|G|}}d+p^t\deg_K(D') \\
\ & =\ds{\frac{p^t}{|G|}}(nd+|G|\deg_K(D'))=0 \ ,\end{array}\]
hence $p^tD\in i(Div^0(K))$.

\item[(c)] {\em There are isomorphisms $\CaCl^0(E)^G/i^K_E(\CaCl^0(K))\simeq Div^0(E)^G/i(Div^0(K)) \simeq \Z/p^t$.}

\noindent The first isomorpshim is a consequence of the snake lemma sequence of diagram \eqref{DiagBigLemma}.
For the second, by part (b) it is enough to build a divisor of exact order $p^t$. Let $r$ be prime with $p$; by Chebotarev
density theorem there exists a (monic) irreducible polynomial $Q$ in $A$ such that $\deg(Q)=r$ and the prime
$(Q)$ is totally split in $E$. Take a prime $v$ of $K$ dividing $(Q)$ so that, in particular, $\deg_K(v)=r$. Put
\[ \wt{D}:=r\frac{|G|}{p^t}\pr_E - \frac{d}{p^t}i(v)\ ;\]
it is easy to check that (by construction) $\wt{D}\in Div^0(E)^G$ and its order is $p^t$.

\item[(d)] {\em If $\chi\neq \chi_0$, then $e_{\chi}(\wt{D})=0$.}

\noindent Recall that $\Delta=\Gal(F_0/F)$ has order prime to $p$. The previous steps can
be proved exactly in the same way for the field extension $E^\Delta/K^\Delta$, i.e., we have
\[ \frac{Div^0(E^{\Delta})^G}{i^{K^\Delta}_{E^\Delta}(Div^0(K^{\Delta}))}\simeq
\frac{\CaCl^0(E^{\Delta})^G}{i^{K^\Delta}_{E^\Delta}(\CaCl^0(K^{\Delta}))}\simeq\Z/p^t \]
and a generator is the class of
\[ \wt{D'}:=r\frac{|G|}{p^t}\pr_{E^{\Delta}}-\frac{d}{p^t}i^{K^\Delta}_{E^\Delta}(\tilde{v})\in Div^0(E^{\Delta})^G \]
(where $\tilde{v}$ is a prime of $K^\Delta$ lying below $v$). Note that the image of $\wt{D'}$ in
$Div^0(E)$ is
\[ i^{E^{\Delta}}_E(\wt{D'})= r\frac{|G|}{p^t}|\Delta|\pr_E-\frac{d}{p^t}i^{K^\Delta}_E(\tilde{v}) \]
and it still has order $p^t$ because $(|\Delta|,p)=1$. Therefore the class of $i^{E^{\Delta}}_E(\wt{D'})$
generates $\CaCl^0(E)^G/i(\CaCl^0(K))$ and, by construction, $\Delta$ acts trivially on it. Hence for $\chi\neq \chi_0$
we obtain
\[ e_{\chi}(W\otimes \CaCl^0(E))^G = e_{\chi}i(W\otimes \CaCl^0(K)) \]
and for the trivial character we have
\[ e_{\chi_0}\left(\frac{(W\otimes \CaCl^0(E))^G}{i(W\otimes \CaCl^0(K))}\right) \simeq W/p^tW \ .\]
\end{enumerate}

\noindent (3) This is just class field theory. Let $v$ be a place of $K$ which divides
$\infty$ and write $B$ for the ring of elements
in $K$ which are regular outside $v$; since $v$ is of degree 1, we have $\CaCl(B)\simeq \CaCl^0(K)$.
Let $H(K)$ be the maximal abelian unramified extension of $K$ in which $v$ is totally split. By class field theory, the
Artin map provides an isomorphism $\Gal(H(K)/K)\simeq \CaCl(B)$ and, because of the ramification in $E/K$,
we have $H(K)\cap E=K$. Denote by $C$ the integral closure of $B$ in $E$ (i.e., the elements in $E$ which are
regular outside any $w|v$); there is a natural map $\CaCl^0(E) \rightarrow \CaCl(C)\simeq \Gal(H(E)/E)$
which preserves Galois action and is surjective because $\deg_E(w)=1$ ($H(E)$ is the analog of $H(K)$, now totally split at $w$).
It only remains to prove that the natural norm map $\CaCl(C)\rightarrow\CaCl(B)$ is surjective.
By construction $EH(K)\subset H(E)$, hence the restriction map
\[ Res : \Gal(H(E)/E)\rightarrow \Gal(EH(K)/E) \simeq \Gal(H(K)/K) \]
is surjective. The well known diagram of class field theory
\[ \xymatrix { \CaCl(C) \ar[r]^{\simeq\qquad} \ar[d]^N & \Gal(H(E)/E) \ar@{->>}[d]^{Res} \\
\CaCl(B) \ar[r]^{\simeq\qquad} & \Gal(H(K)/K) } \]
concludes the proof.

\noindent (4) Consider the sequence (exact by part (3)\,)
\[ 0\rightarrow (W\otimes \Ker(N))(\chi)\rightarrow (W\otimes \CaCl^0(E))(\chi)
{\buildrel{N}\over{\longrightarrow}}
(W\otimes \CaCl^0(K))(\chi)\rightarrow 0 \ , \]
which yields
\[ |(W\otimes \Ker (N))(\chi)|=\frac{|(W\otimes \CaCl^0(E))(\chi)|}{|(W\otimes \CaCl^0(K))(\chi)|} \]
(we recall that for any $W$-module $M$ one has $|M|=p^{u\cdot \ell_{W}(M)}$ with $\ell_W$ the length and
$u:=[W\otimes\Q_p:\Q_p]$).

\noindent We first assume that  $G=\langle \delta\rangle$ is cyclic. Then, for $\chi\neq \chi_0$, the sequence
(exact by part (2)\,)
\[ 0\rightarrow (W\otimes \CaCl^0(K))(\chi) \rightarrow (W\otimes \CaCl^0(E))(\chi)
{\buildrel{1-\delta}\over{-\!\!\!-\!\!\!-\!\!\!-\!\!\!\longrightarrow}}
(1-\delta)(W\otimes \CaCl^0(E))(\chi)\rightarrow 0 \]
yields $|(1-\delta)(W\otimes \CaCl^0(E))(\chi)|=|(W\otimes \Ker(N))(\chi)|$ (simply by counting cardinalities).
Since $(1-\delta)(W\otimes \CaCl^0(E))(\chi)\subseteq (W\otimes \Ker(N))(\chi)$, we have the equality between them.

\noindent For the general case $G$ we use an induction argument on $|G|$. If $|G|=1$ there is nothing to prove
(or, if $|G|=p$, then $G$ is cyclic and we have the proof above). Consider now $K\subset E'\subset E$, where
$G_1:=\Gal(E/E')$ is a cyclic group and we put $G_2:=\Gal(E'/K)$. By the inductive step and the cyclic case
we have
\[ (W\otimes \Ker(N^{E'}_K))(\chi)=I_{G_2}(W\otimes \CaCl^0(E'))(\chi) \]
and
\[ (W\otimes \Ker(N^{E}_{E'}))(\chi)=I_{G_1}(W\otimes \CaCl^0(E))(\chi) \ .\]
By part (3) all norms are surjective and, since $N=N^{E'}_{K}\circ N^E_{E'}$, we have
\[ \begin{array}{cl} (W\otimes \Ker(N))(\chi) & =(N^E_{E'})^{-1}I_{G_2}(W\otimes \CaCl^0(E'))(\chi) \\
\ & = I_G (W\otimes \CaCl^0(E))(\chi)+(W\otimes \Ker(N^E_{E'}))(\chi) \\
\ & = I_G (W\otimes \CaCl^0(E))(\chi)\ , \end{array} \]
because $(W\otimes \Ker(N^E_{E'}))(\chi)=I_{G_1}(W\otimes \CaCl^0(E))(\chi)\subset
I_{G}(W\otimes \CaCl^0(E))(\chi)$.
\end{proof}

\begin{proof}[Proof of Theorem \ref{Cfinitgen}]
Recall that $C_n:=\calCl(F_n)\{p\}$: by Corollary \ref{cor3.4}, we know that for $\chi\neq \chi_0$
\[ \Fitt_{W[\G_n]}((W\otimes C_n)(\chi))=(\Theta_n^{\#}(1,\chi))\ .\]
By the previous lemma the kernel of $N(C)^{n+1}_n:(W\otimes C_{n+1})(\chi)\twoheadrightarrow (W\otimes C_n)(\chi)$
is $I_{\G^{n+1}_n}(W\otimes C_{n+1})(\chi)$ and we know that $\gotI_n=\plim{m}I_{\G^{m+n}_n}$. Hence
\[ \mathcal{C}(\chi)/\gotI_n\mathcal{C}(\chi) \simeq (W\otimes C_n)(\chi) \]
as $W[\G_n]$-modules.
The generalized version of Nakayama Lemma implies that $\mathcal{C}(\chi)$ is a
finitely generated $\L$-module because $(W\otimes C_n)(\chi)$ is a finitely generated $W[\G_n]$-module.
Now simply recall that $\Theta_n^{\#}(1,\chi)((W\otimes C_n)(\chi))=0$ and that
$\Theta_{\infty}^{\#}(1,\chi)=\plim{n}\Theta_n^{\#}(1,\chi)$ to get
\[ \Theta_{\infty}^{\#}(1,\chi)\mathcal{C}(\chi) =
\plim{n}\Theta_n^{\#}(1,\chi)\left( \plim{n} (W\otimes C_n)(\chi) \right) = 0 \ ,\]
i.e., $\mathcal{C}(\chi)$ is a torsion $\L$-module.
\end{proof}

\begin{proof}[Proof of Theorem \ref{IMClevel} {\rm (IMC)}]
By Theorem \ref{Cfinitgen} the Fitting ideal $\Fitt_{\L}(\calc(\chi))$ is well defined. The statement
is equivalent to the equality
\[ \Fitt_{\L}\left( \plim{n} (W\otimes C_n)(\chi) \right)=
\plim{n}\left( \Fitt_{W[\G_n]}((W\otimes C_n)(\chi)) \right)\ . \]
\noindent Take $r\in\N$ such that $e_{\chi}\mathcal{C}$ is generated by $r$ elements,
consider the following commutative diagram
\[ \xymatrix { B_{n+1} \ar@{^(->}[r] \ar[d]^{b^{n+1}_n} &  W[\G_{n+1}]^r \ar@{->>}[r] \ar[d]^{\pi^{n+1}_n}
& (W\otimes C_{n+1})(\chi) \ar[d]^{N(C)^{n+1}_n} \\
B_n \ar@{^(->}[r] &  W[\G_n]^r \ar@{->>}[r] & (W\otimes C_n)(\chi) } \]
(where the central vertical map is the natural projection) and note that the kernels of $\pi^{n+1}_n$ and $N(C)^{n+1}_n$
are respectively $\left( I_{\G^{n+1}_n}W[\G_{n+1}]\right)^r$ and $I_{\G^{n+1}_n}(W\otimes C_{n+1})(\chi)$ (by Lemma
\ref{BigLemma} part (4)). Therefore the induced map between the kernels is surjective and this, together with the
surjectivity of $\pi^{n+1}_n$, yields the surjectivity of $b^{n+1}_n$ by the snake lemma.

\noindent Now the diagram above verifies the Mittag-Leffler
condition, so, taking the limit, we have an exact sequence
\begin{equation}\label{EqInfLev}
B_\infty:=\plim{n} B_n \iri \L^r \sri \mathcal{C}(\chi) \ .
\end{equation}
Working as in Theorem \ref{FittCalmCalf}, one sees that the surjectivity of the $b_n$ and the sequence \eqref{EqInfLev} yield
\[ \begin{array}{cl} \Fitt_{\L}(\mathcal{C}(\chi)) & = \ds{\plim{n}}\left( \Fitt_{W[\G_n]}((W\otimes C_n)(\chi)) \right) \\
\ & =\ds{\plim{n}} \left( \Theta_n^\#(1,\chi) \right) = (\Theta_\infty^\#(1,\chi))\ .\qquad\qquad\qquad\qquad\qquad\qedhere\end{array} \]
\end{proof}

\noindent We end this section with a remark on the module structure of $\mathcal{C}(\chi)$ which depends on the injectivity
of the inclusion maps (i.e., part (1) of Lemma \ref{BigLemma}); the proof is similar to \cite[Proposition 13.28]{Wash} (which
depends on the injectivity of \cite[Proposition 13.26]{Wash}).

\begin{prop}\label{lem4.1}
The $\L$-module $\mathcal{C}(\chi)$ has no nontrivial finite $\Lambda$-submodule.
\end{prop}

\begin{proof} Let $M$ be a finite $\Lambda$-submodule of $\mathcal{C}(\chi)$ of order $|M|=s$ (obviously a power of $p$).
It is enough to prove that there is no $p$-torsion in $M$, so let $\alpha=(\alpha_n)_{n\in \N}\in M$ be
such that $p\alpha=0$ (so that $p\alpha_n=0$ for any $n\gg 0$). Fix $n$ and take a $\g_{i,j}$ among the generators
of $\G$ which acts trivially on $F_n$. Denote by $L_{\infty}=\cup L_m$ the $\Z_p$-extension topologically
generated by $\g_{i,j}$ over $F_0$. The $s+1$ elements of $M$
\[ \alpha\ ,\ \g_{i,j}^p\alpha\ ,\ \g_{i,j}^{p^2}\alpha\ ,\ \dots\ ,\ \g_{i,j}^{p^s}\alpha \]
cannot be distinct, hence there exist $0\leq r<t\leq s$ such that \
\[ \g_{i,j}^{p^r}\alpha=\g_{i,j}^{p^t}\alpha\quad,\ {\rm i.e.,}\quad \g_{i,j}^{p^r}\left(1-\g_{i,j}^{p^t-p^r}\right)\alpha =0\ .\]
This yields $\g_{i,j}^{p^r(p^{t-r}-1)}\alpha =\alpha$: since $\g_{i,j}$ and $\g_{i,j}^{p^{t-r}-1}$ generate the same
$\Z_p$-extension we can assume from the beginning that there exists an $r\geq 0$ such that $\g_{i,j}^{p^r}\alpha=\alpha$.
By construction $F_n\cap L_{\infty}=F_0$ and, for any $m\geq r$,
$\Gal(L_{m+1}/L_m)$ (generated by $\g_{i,j}^{p^m}$) acts trivially on $\alpha$.
Take $\nu$ big enough to have $p\alpha_\nu=0$ and $L_{m+1}F_n\subset F_\nu$, and consider the tower of extensions
\[ F_n\subset L_mF_n\subset L_{m+1}F_n\subset F_\nu \ .\]
From the surjectivity of the norm maps proved in Lemma \ref{BigLemma} one has
\[ \mathcal{C}(\chi)=\plim{[L:F_0]<\infty} (W\otimes C_L)(\chi) \ ,\]
so we can compute
\[ i_\nu^n(\alpha_n)=i_\nu^n(N^\nu_n(\alpha_\nu))
=i_\nu^n(N^{L_mF_n}_{F_n}(N^{L_{m+1}F_n}_{L_mF_n}(N^{F_\nu}_{L_{m+1}F_n}(\alpha_\nu)))) \ .\]
Since $\Gal(L_{m+1}F_n/L_mF_n)$ acts trivially on $\alpha$, the norm $N^{L_{m+1}F_n}_{L_mF_n}$ is just
multiplication by $p$ and we get $i_\nu^n(\alpha_n)=0$.
But Lemma \ref{BigLemma} part (1) shows that the maps like $i$ are injective so $\alpha_n=0$ and, eventually,
$\alpha=0$ as well.
\end{proof}

\section{Application to Bernoulli-Goss numbers and $\pr$-adic $L$-functions}\label{SecArithAppl}
We define an arithmetic invariant related to our $\pr$-adic $L$-function.

\begin{defin}\label{Defm(i)}
For any $i$, define
\begin{equation}\label{eqinvmp}
m_{\pr}(i):=\left\{\begin{array}{cl}
\Inf \left\{ v_\pr(L_{\pr}(1,y,\omega_\pr^i))\,:\ y\in\Z_p\,\right\} & {\rm for}\ i\not\equiv 0\pmod{q-1}\\
\ & \\
\Inf \left\{ v_\pr\left(\ds{\frac{d}{dX}}L_{\pr}(X,y,\omega_\pr^i)_{|X=1}\right)\,:\ y\in\Z_p\, \right\} &
{\rm for}\ i\equiv 0\pmod{q-1}
\end{array}\right. \ .\end{equation}
Obviously the value of $m_{\pr}(i)$ depends only on the class of $i$ modulo $q^d-1$.
\end{defin}

\begin{lem}\label{lem5padicLfun} We have the following equality
\[ m_{\pr}(i)=\Inf \left\{ v_\pr(\beta(j))\,:\ j\geq 1,\ j\equiv i\pmod{q^d-1}\,\right\}\ .\]
\end{lem}

\begin{proof} Let $j\equiv i\pmod{q^d-1}$, then, by Corollary \ref{cor2padicLfun},
\[ L_{\pr}(X,j,\omega_\pr^i)=(1-\pi_\pr^jX^d)Z(X,j) \ .\]
By Definition \ref{DefZXjBGPol}, if $j\not\equiv 0\pmod{q-1}$, we have
\[ (1-\pi_\pr^j)\beta(j)=(1-\pi_\pr^j)Z(1,j)=L_{\pr}(1,j,\omega_\pr^i) \ ,\]
while, if $j\geq 1$ with $j\equiv 0\pmod{q-1}$, we have
\[ \frac{d}{dX}L_{\pr}(X,j,\omega_\pr^i)_{|X=1}=-(1-\pi_\pr^j)\frac{d}{dX}Z(X,j)_{|X=1} =
-(1-\pi_\pr^j)\beta(j) \]
(recall $Z(1,j)=0$ in this case). The lemma follows noting that
the set $\{j\geq 1\,,\ j\equiv i\pmod{q^d-1}\,\}$ is dense in $\Z_p$.
\end{proof}

We can now prove a function field version of the Ferrero-Washington Theorem (see, e.g., \cite[Theorem 7.15]{Wash}),
but its statement is limited to nontrivial characters.

\begin{thm}\label{Sticnon0modp}
For any $1\leq i\leq q^d-2$, one has $\Theta_\infty^\#(1,\wt{\omega}_\pr^i)\not\equiv 0\pmod p$.
\end{thm}

\begin{proof}
We consider two cases depending on the type of the character $\wt{\omega}_\pr^i$.
Recall that, by Lemma \ref{BetaCongrTheta}, $\beta(j)\neq 0$ for any $j\geq 0$.

\noindent {\bf Case 1:} {\em $i\not\equiv 0\pmod{q-1}$, i.e., $\wt{\omega}_\pr^i$ is odd.}\\
In this case $\Theta_\infty^\#=\Theta_\infty$. The previous lemma
shows that $L_{\pr}(1,j,\omega_\pr^{-i})$ is nonzero, hence, by
Theorem \ref{StickLfun},
$s_X(\Theta_\infty(X,\wt{\omega}_\pr^i))(-j)_{|X=1}$ is nonzero as
well, in particular one finds out that
$s(\Theta_{\infty}(1,\wt{\omega}_{\pr}^i))\neq 0$. Since the map $s$
is injective on $\L/p\L$ (by Theorem \ref{teosmap}), it follows that
\[ \Theta_\infty(1,\wt{\omega}_\pr^i)\not\equiv 0\pmod p \ .\]

\noindent {\bf Case 2:} {\em $i\equiv 0\pmod{q-1}$, i.e.,
$\wt{\omega}_\pr^i$ is even but $\neq \chi_0$.}\\
In this case $\Theta_\infty^\#=\ds{\frac{\Theta_\infty}{1-X}}$, hence
\[ \Theta_\infty^\#(1,\wt{\omega}_\pr^i) = \frac{d}{dX}\Theta_\infty(X,\wt{\omega}_\pr^i)_{|X=1}\ .\]
The previous lemma shows that
\[ \frac{d}{dX}L_{\pr}(X,j,\omega_\pr^{-i})_{|X=1}\neq 0 \]
and Theorem \ref{StickLfun} yields
\[ s_X\left(\frac{d}{dX}\Theta_\infty(X,\wt{\omega}_\pr^i)\right)(-j)=\frac{d}{dX}s_X(\Theta_\infty(X,\wt{\omega}_\pr^i))(-j)=
\frac{d}{dX}L_{\pr}(X,j,\omega_\pr^{-i}) \ .\]
We get the claim as in the previous case: specializing at $X=1$ and using the injectivity of $s$.
\end{proof}

As a consequence we find that the $p^n$-torsion of $\calc(\chi)$ looks like a pseudo-null module in the non-noetherian Iwasawa
algebra $\L$.

\begin{cor}
For any character $\chi\neq \chi_0$, $p$ does not divide $\Fitt_{\L} (\calc(\chi))$ and the $p^n$-torsion
modules $\calc(\chi)[p^n]$ have at least two relatively prime annihilators.
\end{cor}

\begin{proof}
Easy consequences of the previous theorem and Theorem \ref{IMClevel}
\end{proof}

\begin{rems}
\begin{itemize} \item[{}]
\item[{\bf 1.}] Since we are working in the non-noetherian algebra $\L$, the module $\mathcal{C}(\chi)[p^\infty]$
may not be finitely generated on $W$. The last statement (recalling pseudo-nullity for noetherian Iwasawa
algebras) might not be true if we consider the whole set of $p$-power torsion points $\mathcal{C}(\chi)[p^\infty]$.
However a combination of Proposition \ref{lem4.1} and the previous corollary suggests
to investigate the possibility that $\mathcal{C}(\chi)[p^\infty]=0$.
\item[{\bf 2.}] In \cite[page 4446]{GuoShu} the authors provide a formula for the class number growth in subextensions of the
$\pr$-cyclotomic extension and note that the growth can be exponential, i.e., the direct analog of the Ferrero-Washington
Theorem ($\mu=0$) does not hold for function fields.
\end{itemize}
\end{rems}

An estimate for $m_\pr(i)$ is provided by the following

\begin{lem}\label{Estmi} For any $1\leq i\leq q^d-2$ with $i\not\equiv 0\pmod{q-1}$, one has
\[ m_{\pr}(i) \leq \frac{i}{d}\log_q(i+1)\ .\]
\end{lem}

\begin{proof} By Lemma \ref{Simon}, we have
\[ S_n(j)=\sum_{a\in A_{+,n}}a^j=0 \quad{\rm for\ any}\ n > \log_q(j+1) \ .\]
Hence
\[ \beta(j)=\sum_{n\geq 0} S_n(j) = 1+\sum_{n=1}^{\left[ \log_q(j+1) \right]}
\sum_{a\in A_{+,n}}a^j \]
(where $[*]$ means the integral part of $*$). Clearly we have
\[ \deg(\beta(j)) \leq \left[\log_q(j+1)\right]j \]
and the result follows from Lemma \ref{lem5padicLfun}.
\end{proof}

In the classical setting of $p$-adic $L$-functions defined for the cyclotomic $\Z_p$-extension of a number field (see, for example,
\cite[Chapter 5]{Wash}), there is the following natural problem on the $p$-adic valuation of values of $p$-adic $L$-functions
(related to $p$-adic valuations of generalized Bernoulli numbers)

\begin{quest}\label{conj3.1} Let $\chi$ be an even character in $\Hom(\Gal(\Q(\bmu_p)/\Q),\Z_p^*)$,
then is it true that
\[ \Inf\left\{ v_p(L_p(y,\chi))\,:\ y\in\Z_p\,\right\} \leq 1\ ?\]
\end{quest}

In the following we still consider non trivial characters only.

\begin{cor}[Arithmetic properties of Bernoulli-Goss numbers]\label{ArPropBGNum}
Let $1\leq i\leq q^d-2$ and define $N_\pr(i):=\Inf\{n\geq 0\,:\,\Theta_n^\#(1,\wt{\omega}_\pr^i)\not\equiv 0\pmod{p}\,\}$,
which is well defined because of Theorem \ref{Sticnon0modp}. Then
\[ N_\pr(i)\leq \Inf \{ v_\pr(\beta(j))\,:\ j\geq 1,\,\ j\equiv -i\pmod{q^d-1}\,\} = m_\pr(q^d-1-i) = m_\pr(-i) \ .\]
\end{cor}

\begin{proof}
Assume $i\not\equiv 0\pmod{q-1}$: by definition of $m_\pr(i)$ (or by Lemma \ref{lem5padicLfun}),
there exists $y_0\in\Z_p$ such that $L_\pr(1,y_0,\omega_\pr^{-i})\not\equiv 0\pmod{\pr^{m_\pr(-i)+1}}$, while
for any $y\in\Z_p$ we have $L_\pr(1,y,\omega_\pr^{-i})\equiv 0\pmod{\pr^{m_\pr(-i)}}$.

\noindent In Section \ref{SecSinn} we saw that the map $s$ can be computed by taking the limit on the (induced)
maps
\[ s_n : W/pW[\G_n]\rightarrow C^0(\Z_p,A_\pr/\pr^{n+1})\ .\]
Therefore, for $n < m_\pr(-i)$, we obtain that $s_n(\Theta_n^\#(1,\wt{\omega}_\pr^i))$ is the zero function,
while
\[ s_{m_{\pr}(-i)}(\Theta_{m_\pr(-i)}^\#(1,\wt{\omega}_\pr^i))(-y_0)\neq  0 \ .\]
Since the maps $s_n$ are not injective in general (see Proposition
\ref{snNotInj}) we only obtain an inequality
\[ N_\pr (i)=\Inf\{n\geq 0 \,:\ \Theta_n^\#(1,\chi)\not\equiv 0\pmod{pW[\G_n]}\,\} \leq m_\pr(-i) \ .\]
The proof for even nontrivial characters (i.e., for $i\equiv 0\pmod{q-1}$, $i\neq 0$) is similar.
\end{proof}

\begin{rem} We can define similar arithmetic invariants for the $\Z_p$-cyclotomic extension of a number field $k$.
For simplicity we just consider $k=\Q$ with $p\neq 2$, the generalization is straightforward.
Take a character $\chi$ in $\Hom(\Gal(\Q(\bmu_p)/\Q),\Z_p^*)$ and let
\[ L_p(y,\chi)=f((1+p)^y-1,\chi)\ ,\ \ f(T,\chi)\in \Z_p[[T]]\]
be the associated $p$-adic $L$-function. Let $\Q_n:=\Q(\boldsymbol{\mu}_{p^{n+1}})$ be the layers of the $\Z_p$-cyclotomic
extension and let $\CaCl(\Q_\infty):=\plim{n} \CaCl(\Q_n)\otimes_{\Z}\Z_p$. The Iwasawa Main Conjecture in this setting
(Mazur-Wiles Theorem \cite{MW}, for an overview of Rubin's proof using Kolyvagin's method see, e.g., \cite[Chapter 15]{Wash})
reads as
\[ \Fitt_{\Z_p[[T]]} (\CaCl(\Q_\infty)(\omega_\pr\chi^{-1}))=(f(T,\chi)) \ ,\]
and we can define
\[ m_p(\chi):=\Inf \{\,v_p(L_p(y,\chi))\,:\ y\in\Z_p\,\} \in \N \ .\]
Finally let $\Theta_{\Q_n/\Q,p}(\chi)$ be the Stickelberger element (see \cite[Chapter 6]{Wash}): again via the
Main Conjecture, we have
\[ \Fitt_{\Z_p[\Gal(\Q_n/\Q)]}(\CaCl(\Q_n)(\omega_\pr\chi^{-1}))=(\Theta_{\Q_n/\Q,p}(\chi)) \ .\]
The Ferrero-Washington Theorem implies that
\[ N_p(\chi):=\Inf\{\,n\geq 0\,:\ \Theta_{\Q_n/\Q,p}(\chi)\not\equiv 0\pmod{p}\,\} \]
is well defined.

\noindent At present it is not clear whether there is any kind of relation between $m_p(\chi^{-1})$, the $p$-adic valuations
of generalized Bernoulli numbers and $N_p(\chi)$ in this setting.
\end{rem}

\appendix
\begin{section}{Computing Fitting ideals with a totally ramified place}
This appendix generalizes some of the results in \cite[Sections 2
and 3]{GP2} (with $\ell=p$) to the case in which there is a totally
ramified place $v$ not necessarily $\F_q$-rational.

\noindent Let $X\rightarrow Y$ be a finite abelian Galois covering of smooth projective curves over a finite field $\F$
(of order $q=p^k$ with $p$ prime) with Galois group $G$ and unramified outside a finite set of places $S$.
Equivalently we can consider a finite abelian Galois extensions
$L/F$ (where $L$ and $F$ are the function fields $\F(X)$ and $\F(Y)$ respectively) with Galois group $G$. We always assume
that $\F$ is the field of constants of $X$, i.e., the covering $X/Y$ is geometric. Let $S$ be a non-empty finite set
of closed points on $X$ containing all the ramified primes of $X/Y$.

\noindent We recall the short exact sequence of $\Z_p[G][[G_{\F}]]$-modules (analogous to \eqref{SeqL})
\begin{equation}\label{eqap0}
0\rightarrow T_p(Jac(X)(\overline{\F}))\rightarrow T_p(\calm_{\ov{S}})\rightarrow L\rightarrow 0 \ ,
\end{equation}
where $Jac(X)(\overline{\F})$ are the $\ov{\F}$-points of the Jacobian of $X$, $T_p(Jac(X)(\ov{\F}))$ is the usual Tate
module at $p$ and $L$ is the kernel of the degree map $\Z_p[\ov{S}]\rightarrow \Z_p$ ($\ov{S}$ is the set of points in
$X\times_{\F}\ov{\F}$ above the points $v\in S$). The central module $T_p(\calm_{\ov{S}})$ is related to the $p$-adic
realization (or $p$-adic Tate module) associated to the Picard 1-motive $M_{\ov{S},\ov{\Sigma}}$
($\Sigma$ is a set of closed points of $X$ disjoint from $S$): a detailed description of $M_{\ov{S},\ov{\Sigma}}$
in terms of divisor classes quite useful for computations is provided in \cite[Section 2]{GP1}.
For any set $\Sigma$ there is a short exact sequence (see \cite[Equation (2)]{GP2})
\[ 0\rightarrow T_p(\tau_{\Sigma}(\overline{\F}))\rightarrow T_p(\calm_{\ov{S},\ov{\Sigma}})\rightarrow
T_p(\calm_{\ov{S}})\rightarrow 0 \ ,\]
but, since we are only considering the case $p=char(F)$, the toric part $T_p(\tau_{\Sigma}(\ov{\F}))$
vanishes (\cite[Remark 2.7]{GP1}) and $\Sigma$ has no concrete influence on the module we are interested in,
i.e., we can assume $\Sigma=\emptyset$. In this setting the main result of \cite{GP1} (\cite[Theorem 4.3]{GP1},
see also \cite[Lemma 2.3]{GP2}) reads as follows

\begin{thm}[Greither-Popescu]\label{teoap1} The Tate module $T_p(\calm_{\ov{S}})$ is cohomologically trivial
over $G$ hence (since it is free of finite rank over $\Z_p$)
projective over $\Z_p[G]$. Moreover the Fitting ideal of
$T_p(\calm_{\ov{S}})$ over $\Z_p[G][[G_\F]]$ is principal and generated
by $\Theta_{L/F,S}(\g^{-1})$, where $\g$ is the arithmetic Frobenius
in $G_\F=\Gal(\ov{\F}/\F)$.
\end{thm}

\noindent Assume once and for all that there exists a point $v_1\in S$ which is totally ramified in $X/Y$.
For a place $v\in Y$ denote by $G_v$ the decomposition subgroup of $v$ in $G$ and by $I_v$ its inertia
subgroup. Put $H_v:=\Z_p[\ov{X}(v)]$ (where $\ov{X}(v)$ denote the set of
points of $X\times_{\F}\overline{\F}$ above $v$): it is a $\Z_p[G][[G_\F]]$-module and we observe that
$\Z_p[\ov{S}]=\oplus_{v\in S} H_v$. Let $\Fr_v$ denote the Frobenius of $v$ in $G$ (if $v$ is unramified,
in the ramified case any lift of a Frobenius of $v$ in $G_v/I_v$ will do) and put
$e_v(u):=1-\Fr_v^{-1}u^{d_v}\in\Z[G][u]$ (where $d_v$ is the degree of the place $v$).

\noindent We recall \cite[Lemmas 2.1 and 2.2]{GP2}.

\begin{lem} $H_v$ is a cyclic $\Z_p[G][[G_\F]]$-module and we have:
\begin{enumerate}
\item if $v\in Y$ is unramified in $X$, then
\[ \Fitt_{\Z_p[G][[G_\F]]}(H_v)=(e_v(\g^{-1}))\quad {\rm and}\quad H_v\simeq \Z_p[G][[G_\F]]/(e_v(\g^{-1}))\ ;\]
\item if $v\in Y$ is ramified in $X$, then
$\Fitt_{\Z_p[G][[G_\F]]}(H_v)=(e_v(\g^{-1}),\tau-1\,:\,\tau\in I_v)$ and
\[ H_v\simeq \Z_p[G][[G_\F]]/(e_v(\g^{-1}),\tau-1\,:\,\tau\in I_v)\simeq \Z_p[G/I_v][[G_\F]]/(e_v(\g^{-1})) \ .\]
\end{enumerate}
\end{lem}

\noindent Now we give a closer look at the $\Z_p[G][[G_\F]]$-module
$L$ in order to compute the Fitting ideal of $T_p(Jac(X)(\ov{\F}))$.

\begin{lem}\label{lemA3} Let $X/Y$ be a geometric Galois cover with a totally ramified prime $v_1\in S$. Then
 we have an isomorphism of $\Z_p[G][[G_{{\F}}]]$-modules:
\begin{equation}\label{eqap1}
L\simeq (\g-1)H_{v_1}\oplus(\oplus_{v\in S'} H_v)
\end{equation}
where $S':=S-\{v_1\}$.
\end{lem}

\begin{proof} Note that the primes in $\ov{X}(S)$ are points in $X\times_{\F}\ov{\F}$ and
$\deg : H_v\rightarrow\Z_p$ is surjective.

Since $v_1$ is totally ramified we have that $(\gamma-1)H_{v_1}+
\Z_p 1_{H_{v_1}}$ is a $\Z_p[G][[G_{\F}]]$-submodule of
$H_{v_1}\otimes_{\Z_p}\Q_p$, and note that $\Ker\{\deg :
H_{v_1}\rightarrow \Z_p\}=(\g-1)H_{v_1}$, because total
ramification yields $H_{v_1}\simeq \Z_p[[G_\F]]/(1-\g^{-d_{v_1}})$.

We have an injective morphism of $\Z_p[G][[G_{\ov{\F}}]]$-modules:
\[ (\g-1)H_{v_1}\oplus(\oplus_{v\in S'} H_v)\hookrightarrow
(((\gamma-1)H_{v_1}+ \Z_p 1_{H_{v_1}})\oplus(\oplus_{v\in S'} H_v)) \]
given by
\[ (\alpha,\beta)\mapsto (\alpha-deg(\beta)1_{H_{v_1}},\beta) \ .\]
Since all points of $X\times_{\F}\ov{\F}$ above $v_1$ have degree 1, the
degree map on $H_{v_1}$ (via the identification with
$\Z_p[[G_{\F}]]/(e_{v_1}(\g^{-1}))$) sends $1_{H_{v_1}}$ to a unit
in $\Z_p$. Hence the image of the injective morphism above is
inside $L$, and its image is exactly $L$ because of the
decomposition $H_{v_1}=(\gamma-1)H_{v_1}\oplus\Z_p 1_{H_{v_1}}$ (or
checking $\Z_p$-ranks).
\end{proof}

Now we can generalize \cite[Theorem 2.6]{GP2}

\begin{thm}\label{teoap2} Suppose $X/Y$ a geometric Galois cover with a totally ramified point $v_1\in S$
of degree $d_{v_1}$. Then
\begin{equation}\label{AppFitt}
\Fitt_{\Z_p[G][[G_\F]]}(T_p(Jac(X)(\ov{\F}))^*)=(\Theta_{L/F,S}(\g^{-1}))\cdot
\left(1,\frac{n(G)}{\frac{1-\g^{-d_{v_1}}}{1-\g^{-1}}}\right)\cdot
\prod_{\begin{subarray}{c} v\in S' \\v\neq v_1\end{subarray}}
\left( 1,\frac{n(I_v)}{e_v(\g^{-1})}\right)\ ,
\end{equation}
where $^*$ denotes the $\Z_p$-dual and $n(G):=\sum_{\sigma\in G}\sigma\in\Z[G]$.
\end{thm}

\begin{proof} By Lemma \ref{lemA3}, $L\simeq (\g-1)H_{v_1}\oplus\left(\oplus_{v\in S'}H_v\right)$ as a
$\Z_p[G][[G_\F]]$-module. For any $v\in S$ we have the following short exact sequence
(\cite[after Lemma 2.4]{GP2}):
\begin{equation}\label{eqap2}
0\rightarrow H_v\rightarrow
\Z_p[G][[G_{\F}]]/e_{v}(\g^{-1})\rightarrow
\Z_p[G][[G_{\F}]]/(e_v(\g^{-1}),s(I_v))\rightarrow 0\ .
\end{equation}
 We now obtain a short exact sequence for the
$\Z_p[G][[G_\F]]$-module $(\g-1)H_{v_1}$ (which is not trivial if
$d_{v_1}>1$). Write
\[ (\g-1)H_{v_1}=(\g e_1-e_1)\Z_{p}+\cdots+(\g^{d_{v_1}-1}e_1-\g^{d_{v_1}-2}e_1)\Z_{p} \]
as a free $\Z_p[G]$-module of rank $d_{v_1}-1$, and put $w_{i}:=\g^ie_1-\g^{i-1}e_1$. With respect to the basis
$w_1,\ldots,w_{d_{v_1}-1}$ and applying \cite[Proposition 2.1]{GP1} we compute
\[ \det(id- \g u\,|\,(\g-1)H_{v_1})=1+u+u^2+\ldots+u^{d_{v_1}-1}=\frac{1-u^{d_{v_1}}}{1-u} \ .\]
Therefore one has a short exact sequence
\begin{equation}\label{eqap3}
(\g-1)H_{v_1} \iri \Z_p[G][[G_\F]]/\left(\frac{1-\g^{-d_{v_1}}}{1-\g^{-1}}\right) \sri
\Z_p[G][[G_\F]]/\left(\frac{1-\g^{-d_{v_1}}}{1-\g^{-1}},n(G)\right)
\end{equation}
(resembling \eqref{ResHpr}). We now proceed as in Case 3 of Lemma \ref{FittTFn}:
putting together equations \eqref{eqap0}, \eqref{eqap2} and \eqref{eqap3} we obtain the four term exact sequence
\[ T_p(Jac(X)(\ov{\F})) \iri T_p(\calm_S) \rightarrow
\bigoplus_{v\in S'}\Z_p[G][[G_\F]]/e_v(\g^{-1})\oplus
\Z_p[G][[G_\F]]/\left(\frac{1-\g^{-d_{v_1}}}{1-\g^{-1}}\right) \]
\vspace{-1truecm}
\[ \hspace{4.5truecm} \xymatrix{ \ar@{->>}[d] \\ \ } \]
\vspace{-.5truecm}
\[ \hspace{3truecm} \bigoplus_{v\in S'}\Z_p[G][[G_\F]]/(e_v(\g^{-1},n(I_v))\oplus
\Z_p[G][[G_\F]]/\left(\frac{1-\g^{-d_{v_1}}}{1-\g^{-1}},n(G)\right)\ . \]
Denote by $X_2$, $X_3$ and $X_4$ the second, third and fourth modules appearing in the sequence above.
All modules are finitely generated and free over $\Z_p$, moreover $X_3$ has projective dimension
0 or 1 over $\Z_p[G][[G_\F]]$, while $X_2$ has projective dimension 1 over $\Z_p[G][[G_\F]]$ because
it has no non-trivial finite submodules (enough by \cite[Proposition 2.2 and Lemma 2.3]{Po})
and is $G$-cohomologically trivial by Theorem \ref{teoap1}. With these properties \cite[Lemma 2.4]{GP2}
yields
\[ \Fitt_{\Z_p[G][[G_\F]]}(T_p(Jac(X)(\ov{\F}))^*)\Fitt_{\Z_p[G][[G_\F]]}(X_3) \!\! = \!\!
\Fitt_{\Z_p[G][[G_\F]]}(X_2)\Fitt_{\Z_p[G][[G_\F]]}(X_4).\]
Since
\[ \Fitt_{\Z_p[G][[G_\F]]}(X_3)=\left(\frac{1-\g^{-d_{v_1}}}{1-\g^{-1}}\right)\cdot\left(\prod_{v\in S'}e_v(\g^{-1})\right) \]
and
\[ \Fitt_{\Z_p[G][[G_\F]]}(X_4)=\left(\frac{1-\g^{-d_{v_1}}}{1-\g^{-1}},n(G)\right)\cdot\prod_{v\in S'}(e_v(\g^{-1}),n(I_v)) \ ,\]
Theorem \ref{teoap1} immediately implies equation \eqref{AppFitt}.\end{proof}

\noindent To apply this to class groups, following \cite{GP2}, observe first that
\[ T_p(Jac(X)(\ov{\F}))^*\simeq \Hom(Jac(X)(\ov{\F})\{ p\},\Q_p/\Z_p) \]
where $Jac(X)(\ov{\F})\{ p\}$ denotes the divisible group given by the $p$-power torsion elements.

\noindent Consider the projection morphism $\pi_{G_\F}:\Z_p[G][[G_\F]]\rightarrow\Z_p[G]$
mapping $\g$ to 1. For a finitely generated $\Z_p[G][[G_\F]]$-module $M$, we have the equality
\[ \pi_{G_\F}\left(\Fitt_{\Z_p[G][[G_\F]]}(M)\right)=\Fitt_{\Z_p[G]}(M_{G_\F}) \]
(where $M_{G_\F}$ denotes the $G_\F$-coinvariant module of $M$).
Moreover the $G_\F$-coinvariants of $T_p(Jac(X)(\ov{\F}))^*$ correspond to
\[ \begin{array}{ll} (T_p(Jac(X)(\ov{\F}))^*)_{G_\F} & \simeq\Hom (Jac(X)(\ov{\F})^{G_\F},\Q_p/\Z_p) \\
\ & = \Hom(Jac(X)(\F),\Q_p/\Z_p)=(\calCl(X))^\vee \end{array}\]
the Pontrjagin dual of the $p$-part of the class group module associated to $X$ (see Remark \ref{Duals}).

\noindent The following generalizes \cite[Theorem 3.2]{GP2}.

\begin{thm}\label{teoap4}
Assume that $X/Y$ is a geometric Galois cover with a totally ramified point $v_1\in S$. Then
\[ \Fitt_{\Z_p[G]}(\calCl(X)\{ p\}^\vee)=\langle\, g_{\mathcal{W}} \cdot
cor^G_{G/I_{\mathcal{W}}}(\Theta_{L^{\mathcal{W}}/F,S - \mathcal{W}}(1))
\,,\,\mathcal{W}\subset S'\,\rangle\ ,\] where $cor$ denotes the
corestriction map, and for any $T\subset S$ we write $I_{T}$ for the
compositum of all inertia groups $I_v$ with $v\in {T}$,
$L^{T}:=L^{I_{T}}$ and $g_{T}=\displaystyle{\frac{\prod_{v\in
T}|I_v|}{|\prod_{v\in T} I_v|}}\in\N$.
\end{thm}

\begin{proof}
Following \cite[p. 232]{GP2}, the Euler relations give us all the
generators for the Fitting ideal (using Theorem \ref{teoap2}), and
we have for any subset $T$ of $S$
\[ \left(\prod_{v\in T}n(I_v)\right)\Theta_{L/F,S}(\g^{-1})=g_T\cdot \prod_{w\in T} e_v(\g^{-1})\cdot
cor^G_{G/I_T}(\Theta_{L^T/F,S-T}(\g^{-1})) \ .\] Moreover, if
$v_1\notin T$ we obtain the relations that appear in Theorem
\ref{teoap2} while, if $v_1\in T$, then
\[ \frac{n(G)}{e_{v_1}(\g^{-1})} \prod_{v\in T-v_1}\frac{n(I_v)}{e_v(\g^{-1})} \Theta_{L/F,S}(\g^{-1})
= g_T\cdot cor^G_{G/I_T} (\Theta_{L^T/F,S-T}(\g^{-1})) \ .\]
Note that $e_{v_1}(\g^{-1})=1-\g^{-d_{v_1}}$ and
consider the element
\[ e_{d_{v_1}}:=\frac{e_{v_1}(\g^{-1})}{\frac{1-\g^{-d_{v_1}}}{1-\g^{-1}}}= 1-\g^{-1}\in\Z_p[G][[G_\F]]^* \ .\]
Multiplying the equalities above by $e_{d_{v_1}}$ we obtain
equalities of ideals in $\Z_p[G][[G_\F]]$.

\noindent The projection map $\pi_{G_\F}$ maps $e_{d_{v_1}}$ to zero
and, using Theorem \ref{teoap2}, we obtain the claim (in particular
no more generators are needed for the Fitting ideal over $\Z_p[G]$
when $v_1\in T$).\end{proof}

\end{section}


\begin{thebibliography}{99}

\bibitem{An} {\it B.~Angl\`es}, On $L$-functions of cyclotomic function fields,
J. Number Theory {\bf 116} (2006), no. 2, 247--269.

\bibitem{AnTa} {\it B.~Angl\`es} and {\it L.~Taelman}, Arithmetic of characteristic $p$ special $L$-values
(with an appendix by V.~Bosser),
to appear in Proc. Lond. Math. Soc., arXiv:1205.2794 [math.NT] (2013).

\bibitem{BaHo} {\it P.N.~Balister} and {\it S.~Howson}, Note on Nakayama's lemma for compact $\Lambda$-modules,
Asian J. Math. {\bf 1} (1997), no. 2, 224--229.

\bibitem{BL} {\it A.~Bandini} and {\it I.~Longhi}, Control theorems for elliptic curves over function fields,
Int. J. Number Theory {\bf 5} (2009), no. 2, 229--256.

\bibitem{BBL1} {\it A.~Bandini}, {\it F.~Bars} and {\it I.~Longhi}, Aspects of Iwasawa theory over function fields,
to appear in the EMS Congress Reports, arXiv:1005.2289 [math.NT] (2010).

\bibitem{BBL2} {\it A.~Bandini}, {\it F.~Bars} and {\it I.~Longhi}, Characteristic ideals and Iwasawa theory,
New York J. Math {\bf 20} (2014), 759--778.

\bibitem{CG} {\it P.~Cornacchia} and {\it C.~Greither}, Fitting ideals of class groups of real fields with prime power conductor,
J. Number Theory {\bf 73} (1998), 459--471.

\bibitem{DoPo} {\it J.~Dodge} and {\it C.~Popescu}, The refined Coates-Sinnot Conjecture for characteristic $p$ global fields,
J. Number Theory {\bf 133} (2013), no. 6, 2047--2065.

\bibitem{goss2} {\it D.~Goss}, $v$-adic Zeta Functions, $L$-series and measures for function fields,
Invent. Math. {\bf 55} (1979), 107--116.

\bibitem{Goss} {\it D.~Goss}, Basic structures of function field arithmetic,
Ergebnisse der Mathematik {\bf 35}, Springer-Verlag, Berlin, 1996.

\bibitem{GrKu} {\it C.~Greither} and {\it M.~Kurihara}, Stickelberger elements, Fitting ideals of class groups of CM-fields
and dualisation,
Math. Z. {\bf 260} (2008), no. 4, 905--930.

\bibitem{GP1} {\it C.~Greither} and {\it C.D.~Popescu}, The Galois module structure of $\ell$-adic realizations of Picard
1-motives and applications,
Int. Math. Res. Not. (2012), no. 5, 986--1036.

\bibitem{GP2} {\it C.~Greither} and {\it C.D.~Popescu}, Fitting ideals of $\ell$-adic realizations of Picard 1-motives and
class groups of global function fields,
J. Reine Angew. Math. {\bf 675} (2013), 223--247.

\bibitem{GuoShu} {\it L.~Guo} and {\it L.~Shu}, Class numbers of cyclotomic function fields,
Trans. Amer. Math. Soc. {\bf 351} (1999), no. 11, 4445--4467.

\bibitem{lltt2} {\it K.F.~Lai}, {\it I.~Longhi}, {\it K.-S.~Tan} and {\it F.~Trihan}, The Iwasawa Main conjecture of constant
ordinary abelian varieties over function fields,
arXiv:1406.6125 [math.NT] (2014).

\bibitem{MW} {\it B.~Mazur} and {\it A.~Wiles}, Class fields of abelian extensions of $\Q$,
Invent. Math. {\bf 76} (1984), 179--330.

\bibitem{Po} {\it C.D.~Popescu}, On the Coates-Sinnott conjecture,
Math. Nachr. {\bf 282} (2009), no. 10, 1370--1390.

\bibitem{rosen} {\it M.~Rosen}, Formal Drinfeld modules,
J. Number Theory {\bf 103} (2003), no. 2, 234--256.

\bibitem{Ro} {\it M.~Rosen}, Number theory in function fields,
{\bf GTM 210}, Springer-Verlag, New York, 2002.

\bibitem{Se} {\it J.-P.~Serre}, Algebraic groups and class fields,
{\bf GTM 117}, Springer-Verlag, New York, 1988.

\bibitem{Si} {\it W.~Sinnott}, Dirichelet series in function fields,
J. Number Theory {\bf 128} (2008), 1893--1899.

\bibitem{Ta1} {\it L.~Taelman}, Special $L$-values of Drinfeld modules,
Annals of Math. {\bf 175} (2012), 369--391.

\bibitem{Ta2} {\it L.~Taelman}, A Herbrand-Ribet theorem for function fields,
Invent. Math. {\bf 188} (2012), 253--275.

\bibitem{Ta} {\it J.~Tate}, Les conjectures de Stark sur les Fonctions $L$ d'Artin en $s=0$,
Progress in Mathematics {\bf 47}, Birkh\"auser, 1984.

\bibitem{thakur} {\it D.S.~Thakur}, Function field arithmetic,
World Scientific Publishing Co., Inc., River Edge, NJ, 2004.

\bibitem{Wash} {\it L.C.~Washington}, Introduction to cyclotomic fields, 2nd ed.,
{\bf GTM 83}, Springer-Verlag, New York, 1997.

\end{thebibliography}
\end{document}